\documentclass{article}

\usepackage{geometry}

\usepackage{algcompatible}

\usepackage{algorithm}
\usepackage{algpseudocode}

\newgeometry{vmargin={30mm}, hmargin={30mm,30mm}}   
\usepackage[utf8]{inputenc}
\usepackage[T1]{fontenc}
\usepackage{lmodern}
\usepackage{amsmath} 
\usepackage{amssymb} 
\usepackage{amsfonts}
\usepackage{amsthm}  
\usepackage{float}
\usepackage{hyperref}
\usepackage[scr]{rsfso}
\usepackage{mathtools}
\usepackage{bm}             
\usepackage{graphicx}       
\usepackage{fancyvrb}       
\usepackage[nottoc]{tocbibind} 
\usepackage{dcolumn}        
\usepackage{booktabs}       
\usepackage{paralist}       
\usepackage[usenames]{xcolor}  

\usepackage{leftidx}
\usepackage{systeme} 

\usepackage{pgfplots}
\usetikzlibrary{intersections, pgfplots.fillbetween}

\usepackage{tikz} 
\usepackage{tikz-3dplot} 
\usetikzlibrary{arrows}
\usepackage{xparse}
\usetikzlibrary{3d}
\usetikzlibrary{calc,intersections,through,backgrounds}
\usetikzlibrary{arrows}
\usepackage{tikz-3dplot} 

\tikzset{
    >=stealth',
    punkt/.style={
           rectangle,
           rounded corners,
           draw=black, thick,
           text width=6.5em,
           minimum height=2em,
           text centered},
    pil/.style={
           ->,
           thick,
           shorten <=2pt,
           shorten >=2pt,}
}

\usepackage{microtype}

\newcommand{\A}{\ensuremath\mathbb{A}}

\newcommand{\R}{\mathbb{R}}

\renewcommand{\S}{\mathbb{S}}

\newcommand{\bl}{\color{black}}

\include{macros}

\theoremstyle{definition}
\newtheorem{df}{Definition}
\theoremstyle{theorem}
\newtheorem{thm}{Theorem}

\newtheorem{lem}{Lemma}
\theoremstyle{theorem}

\theoremstyle{remark}

\theoremstyle{remark}
\newtheorem{rem}{Remark}
\theoremstyle{example}
\newtheorem{ex}{Example}
\theoremstyle{notation}
\newtheorem{notation}{Notation}

\theoremstyle{theorem}

\makeatletter
\renewcommand*\env@matrix[1][*\c@MaxMatrixCols c]{%
  \hskip -\arraycolsep
  \let\@ifnextchar\new@ifnextchar
  \array{#1}}
\makeatother

\title{Sections and Chapters}
\title{Proof complexity of Mal'tsev CSP}

\author{Azza Gaysin\thanks{The major part of this work was carried out during the author's affiliation with the Department of Mathematical Logic, Faculty of Computer Science and Mathematics, at Passau University.}}

\date{}

\begin{document}
\allowdisplaybreaks
\maketitle
\begin{abstract}
Constraint Satisfaction Problems (CSPs) form a broad class of combinatorial problems, which can be formulated as homomorphism problems between relational structures. The CSP dichotomy theorem classifies all such problems over finite domains into two categories: NP-complete and polynomial-time \cite{8104069}, \cite{10.1145/3402029}.

Polynomial-time CSPs can be further subdivided into smaller subclasses. Mal'tsev CSPs are defined by the property that every relation in the problem is invariant under a Mal’tsev operation, a ternary operation $\mu$ satisfying $\mu(x, y, y) = \mu(y, y, x) = x$ for all $x, y$. Bulatov and Dalmau proved that Mal'tsev CSPs are solvable in polynomial time, presenting an algorithm for such CSPs \cite{articlestdfh}. 

The negation of an unsatisfiable CSP instance can be expressed as a propositional tautology. We formalize the algorithm for Mal'tsev CSPs within bounded arithmetic $V^1$, which captures polynomial-time reasoning and corresponds to the extended Frege proof system. We show that $V^1$ proves the soundness of Mal'tsev algorithm, implying that tautologies expressing the non-existence of a solution for unsatisfiable instances of Mal'tsev CSPs admit short extended Frege proofs. 

In addition, with small adjustments, we achieved an analogous result for Dalmau's algorithm that solves generalized majority-minority CSPs - a common generalization of near-unanimity operations and Mal’tsev operations.

\end{abstract}

\section{Introduction}

Constraint Satisfaction Problems (CSPs) encompass a broad class of combinatorial problems. In a CSP instance, the goal is to assign values from a finite domain to a set of variables such that all specified constraints are simultaneously satisfied. A constraint is defined as a pair consisting of a tuple of variables, called scope, and a relation that specifies the allowed combinations of values for the variables in that scope, called constraint relation. If we restrict the set of relations to a set $\Gamma$, called constraint language, we denote an instance of the CSP over that set as $\operatorname{CSP}(\Gamma)$. An equivalent formulation of CSPs is in terms of homomorphisms between relational structures. Given a relational structure $\mathcal{A}$, the problem $\operatorname{CSP}(\mathcal{A})$ asks, for an input structure $\mathcal{X}$ over the same relational vocabulary, whether there exists a homomorphism from $\mathcal{X}$ to $\mathcal{A}$.

Despite their general NP-completeness \cite{8104069}, \cite{10.1145/3402029}, a rich algebraic theory provides precise criteria for tractable subclasses. These criteria are typically phrased in terms of the existence of so-called polymorphisms - operations on the domain that preserve all constraint relations. A well-studied tractable subclass is formed by Mal’tsev CSPs, where the template admits a Mal’tsev polymorphism, that is, a ternary operation $\mu$ on a domain $A$ satisfying $\mu(x, y, y) =\mu(y, y, x)= x$ for all $x,y\in A$. We say that CSP($\mathcal{A}$) is Mal'tsev, if every relation from $\mathcal{A}$ is preserved by a Mal'tsev polymorphism. Mal'tsev CSPs comprise, in particular, affine problems and the general subgroup problem. Bulatov and Dalmau \cite{articlestdfh} proved that Mal'tsev CSPs are solvable in polynomial time by developing an algorithm based on the notion of compact representations. 

When viewed as a homomorphism problem, each instance $\mathcal{X}$ of $\operatorname{CSP}(\mathcal{A})$ can be encoded as a conjunction of a simple and easily interpretable set of clauses, denoted by $\operatorname{CNF}(\mathcal{X}, \mathcal{A})$; see, for example, the constructions in~\cite{10.1145/3265985},~\cite{9579073}. This set is built from propositional atoms $p_{ij}$, one for each element $i$ in the structure $\mathcal{X}$ and each element $j$ in $\mathcal{A}$, representing the possible assignments in a homomorphism between the relational structures. For unsatisfiable instances of $\operatorname{CSP}(\mathcal{A})$, the negation of $\operatorname{CNF}(\mathcal{X}, \mathcal{A})$, which expresses that no homomorphism from $\mathcal{X}$ to $\mathcal{A}$ exists, is a tautology of size polynomial in the sizes of $\mathcal{X}$ and $\mathcal{A}$. This opens the possibility of analyzing tractable subclasses of CSPs using the methods of proof complexity. The question is, in which propositional proof systems such tautologies admit short proofs. 

The first to study this correspondence were Atserias and Ochremiak~\cite{10.1145/3265985}. They studied the impact of standard efficient reductions between constraint languages on the complexity of their propositional proofs. We considered the correspondence for various tractable subclasses of CSPs in \cite{9579073}, \cite{gaysin2023proof} and \cite{gaysin2024}, providing formalizations in different theories of bounded arithmetic. Each theory of bounded arithmetic is associated with a propositional proof system in the sense that, if a universal statement is provable in the theory, then all of its propositional translations admit polynomial-size proofs in the corresponding proof system. 

In this paper, we adapt our framework to Mal’tsev CSPs. Our main result is that the soundness of Mal'tsev algorithm can be proved in a theory of bounded arithmetic $V^1$, which captures polynomial-time reasoning. By soundness, we mean that whenever the algorithm outputs a negative answer, the instance is indeed unsatisfiable. By the known correspondence between the theory and the extended Frege (EF) proof system, it follows that the tautologies $\neg\operatorname{CNF}(\mathcal{X}, \mathcal{A})$ encoding the non-existence of solutions for unsatisfiable instances admit polynomial-size extended Frege proofs. While for satisfiable instances $\mathcal{X}$ of CSP($\mathcal{A}$) Mal'tsev algorithm produces a homomorphism from $\mathcal{X}$ to $\mathcal{A}$ as a witness of a positive answer, for unsatisfiable instances, polynomial-size extended Frege proofs can be regarded as independent witnesses of the correctness of negative answers to $\operatorname{CSP}(\mathcal{A})$. 

In addition, we obtained an analogous result for Dalmau's algorithm~\cite{lmcs:2237}, which solves generalized majority-minority CSPs - a common generalization of near-unanimity operations and Mal’tsev operations. A $k$-ary operation on $A$ is called generalized majority-minority if for all $a,b\in A$, $\varphi(x,y,\ldots,y)= \varphi(y,x,\ldots,y)=\ldots .= \varphi(y,y,\ldots,x)= y$ or $\varphi(x,y,\ldots,y) = \varphi(y,y,\ldots,x) = x$ for all $x,y\in\{a,b\}$. Dalmau's algorithm is very similar to the Mal’tsev algorithm, and its formalization requires only minor adjustments. 

The structure of the paper is as follows. Sections \ref{prelimal} and \ref{alka;hfsdg} review essential concepts from universal algebra, introduce Mal’tsev and generalized majority-minority (GMM) algebras, and present key generating theorems. Sections \ref{as;ldhkgf;} and \ref{al;sdkgh;lggggg} define constraint satisfaction problems (CSPs), with particular focus on Mal’tsev CSPs, and outline the core ideas behind the Mal’tsev algorithm. In Section \ref{==-asl;iyfg;iy}, we provide background from proof complexity and bounded arithmetic, and explain the correspondence between these two areas. Section \ref{asdkjfg;lkg} introduces several general auxiliary functions and relations definable within theory $V^1$. Section \ref{alskdyf;kajdsghy;kgj} is dedicated to the formalization of the Mal’tsev algorithm and is structured in three main parts. Sections \ref{alskdj;lkys}, \ref{askfdh;qweyi}, and \ref{lh;alkdhg;lih;} establish the framework, formalize the necessary concepts from universal algebra, and describe the correspondence between relational structures and algebras. Section \ref{alskhf;shgf} contains the formalization of all algorithmic subroutines along with proofs of their correctness. In Section \ref{alskdjf';lahdg'a}, we formulate and prove the main universal statement and derive the upper bound for Mal’tsev CSPs. Finally, Section \ref{alskdyf;kajdsghy;kgj} extends these results to Dalmau's algorithm and states the corresponding upper bound for GMM CSPs.

\section{Preliminaries}

\subsection{Basic universal algebra notions and results}\label{prelimal}

This subsection draws on results from \cite{10.1145/2677161.2677165} and \cite{barto_et_al}. Several definitions and statements are adapted from~\cite{BurrisSankappanavar1981}. 

    For any non-empty set $A$ and any natural number $n$, a function $f \colon A^n \to A$ is called an $n$-ary \emph{operation} on $A$. An \emph{algebra} $\mathbb{A} = (A, f^\mathbb{A}_1, f^\mathbb{A}_2, \ldots)$ consists of a domain $A$ together with a set of \emph{basic operations} $f^\mathbb{A}_1, f^\mathbb{A}_2, \ldots$ of fixed arities, specified by a signature $\Sigma = \{f_1, f_2, \ldots\}$. For $B\subseteq A$, $\mathbb{B}=(B,f^\mathbb{B}_1, f^\mathbb{B}_2)$ is called a \emph{subalgebra} of $\mathbb{A}$ (denoted $\mathbb{B}\leq\mathbb{A}$) if every basic operation of $\mathbb{B}$ is the restriction of the corresponding operation of $\mathbb{A}$. A \emph{power algebra} $\A^n$ is an algebra with universe $A^n$, where for every symbol $f \in \Sigma$, the operation $f^{\A^n}$ is defined coordinate-wise using $f^{\A}$. We will usually omit the superscripts on operations. A \emph{term} $t$ over a signature $\Sigma$ is a formal expression built from variables using the function symbols in $\Sigma$ through composition. Each such term defines a function $t^{\A}$ on an algebra $\mathbb{A} = (A, f^\mathbb{A}_1, f^\mathbb{A}_2, \ldots)$, known as a \emph{term operation} of $\mathbb{A}$. The collection of all term operations on $\mathbb{A}$ is referred to as the \emph{clone of term operations} of $\mathbb{A}$, and is denoted by $Clone(\mathbb{A})$. An identity is just a pair of terms $s, t$ with the symbol $\approx$ in between them: $s\approx t$. Identities of two expressions that contain exactly one occurrence of an operation symbol on each side are called identities of height $1$. An algebra $\A$ satisfies the identity $s\approx t$, written $\A\vDash s\approx t$, if for all $a_1,\ldots,a_k\in A$, $s^{\A}(a_1,\ldots,a_k) = t^{\A}(a_1,\ldots,a_k)$.

    A \emph{vocabulary} is a finite set of relational symbols $R_1, \ldots, R_n$, each with a fixed arity. A \emph{relational structure} over this vocabulary is a tuple $\mathcal{A} = (A, R_1^\mathcal{A}, \ldots, R_n^\mathcal{A})$, where $A$ is a non-empty set called the \emph{universe} of $\mathcal{A}$, and each $R_i^\mathcal{A}$ is a relation on $A$ with the same arity as the symbol $R_i$. Given relational structures $\mathcal{X}$ and $\mathcal{A}$ over the same vocabulary $R_1, \ldots, R_n$, a \emph{homomorphism} from $\mathcal{X}$ to $\mathcal{A}$ is a mapping $h \colon X \to A$ between their universes such that, for every $m$-ary relation $R_i^\mathcal{X}$ of $\mathcal{X}$ and every tuple $(x_1, \ldots, x_m) \in R_i^\mathcal{X}$, it holds that $(h(x_1), \ldots, h(x_m)) \in R_i^\mathcal{A}$.

    An $m$-ary operation $f: A^m \rightarrow A$ is said to \emph{preserve} an $n$-ary relation $R \subseteq A^n$ (equivalently, $f$ is a \emph{polymorphism} of $R$, or $R$ is \emph{invariant} under $f$) if, for all tuples $(a_{11}, \ldots, a_{1n}), \ldots, (a_{m1}, \ldots, a_{mn}) \in R$, the tuple $(f(a_{11},\ldots .,a_{m1}),\ldots,f(a_{1n},\ldots,a_{mn}))$ also belongs to $R$. Note that for an algebra $\A$ with a term operation $t$, any subalgebra $\R\leq \A^n$ is an invariant relation under $t$. Given a relational structure $\mathcal{A}$, we define $Pol(\mathcal{A})$ as the set of all operations on $A$ that act as polymorphisms of every relation in $\mathcal{A}$. Conversely, for any set of operations $O$ on $A$, we write $Inv(O)$ to denote the set of all relations on $A$ that are preserved by every operation in $O$. More generally, for any set of relations $\Gamma$ on $A$, we write $Pol(\Gamma)$ for the set of all polymorphisms of $\Gamma$.  

    A classical result in universal algebra describes a fundamental connection between algebras and relational structures.

\begin{thm}[\cite{Kibernetika}]\label{fjjduh87}
For any algebra $\mathbb{A}$ there exists relation structure $\mathcal{A}$ such that $Clone(\mathbb{A})$ $ =Pol(\mathcal{A})$.
\end{thm}

\subsection{Mal'tsev and Generalized majority-minority algebras}\label{alka;hfsdg}

In this section, we consider Mal’tsev algebras and generalized majority-minority algebras, and introduce several definitions and results related to them. The definitions, examples, and results are adapted from~\cite{doi:10.1137/050628957} and~\cite{lmcs:2237}.

\begin{df}[Mal’tsev term]
    A ternary operation $\mu:A^3\to A$ on a finite set $A$ is called
\emph{Mal’tsev} if it satisfies $\mu(x, y, y) =\mu(y, y, x) = x$ for all $x,y\in A$.
An algebra that admits a Mal'tsev term is called a \emph{Mal'tsev algebra}.
\end{df}

If we define a field and a multiplicative group as in~\cite{BurrisSankappanavar1981}, the operations $\mu(x, y, z) = x - y + z$ and $\mu(x, y, z) = x y^{-1} z$ respectively are examples of Mal’tsev terms.

Let $A$ be a finite set and $R \subseteq A^n$. Let $i_1, \ldots, i_k \in \{1, \ldots, n\}$. Then the \emph{projection of $R$ onto coordinates $i_1, \ldots, i_k$} is the set 
$\pi_{i_1, \ldots, i_k} R := \{(a_{i_1}, \ldots, a_{i_k}) : (a_1, \ldots, a_n) \in R\}$. For any tuple $(i,a,b)$ with $1 \leq i \leq n$, $a,b \in A$, a \emph{realization} of $(i,a,b)$ is any pair $t_a, t_b \in A^n$ such that 
$$
\pi_{1,\ldots,i-1}(t_a) = \pi_{1,\ldots,i-1}(t_b) \quad \text{and} \quad \pi_i(t_a) = a,\ \pi_i(t_b) = b.
$$
We say that the pair $t_a, t_b$ \emph{witnesses} the tuple $(i,a,b)$.

\begin{df}[Signature]\label{askdjfh;lkasf}
    We define the \emph{signature} $Sig(R)$ of a relation $R\subseteq A^n$ to be the set of all tuples $(i,a,b)$ witnessed by elements in $R$, i.e.:
    $$Sig(R)=\{(i,a,b)\in [n]\times A^2:\exists t_a,t_b\in R \text{ such that }(t_a,t_b)\text{ witnesses }(i,a,b)\}.
    $$
\end{df}

\begin{df}[Representation]
A subset $R'\subseteq R$ is called a \emph{representation} of $R$ if $Sig(R)=Sig(R')$. If, in addition,
$|R'|\leq 2|Sig(R)|$ then $R'$ is called a \emph{compact representation} of $R$. 
\end{df}

Note that every relation $R$ has a compact representation: for every element $(i,a,b) \in \operatorname{Sig}(R)$, it is enough to take only two elements $t_a, t_b$ that witness the tuple. Also, if $R'$ is a minimal representation of $R \subseteq A^n$, then $|R'| \leq 2n \cdot |A|^2$. Indeed, since $i \leq n$ and there are at most $|A|^2$ pairs $(a,b)$, we get that $|R'| \leq 2n \cdot |A|^2$.

\begin{ex}[\cite{doi:10.1137/050628957}]\label{alksdj;asjfegh}
    Let $R$ be $A^n$, and fix some $d \in A$. For any $i \in [n]$ and any $a \in A$, we define an element $e_{i,a}$ as satisfying:
$$
\pi_j(e_{i,a}) := \begin{cases}
a & \text{if } j=i, \\
d & \text{otherwise}.
\end{cases}
$$
Then, for every tuple $(i,a,b)$, the pair $(e_{i,a}, e_{i,b})$ witnesses the tuple. Consequently, the set of elements $\{e_{i,a} : i \in [n],\ a \in A\}$ is a compact representation of the relation $A^n$. 
\end{ex}

\begin{thm}[\cite{doi:10.1137/050628957}]\label{alalksdjasd}
Let $A$ be a finite set, $\A$ be an algebra over $A$ with a Mal'tsev term $\mu$ and $\R \leq \A^n$. Let $R \subseteq \R$ be a representation of $\R$. Then $\R$ is generated by $R$ using only $\mu$. 
\end{thm}

\begin{proof}
Let $\mathbb{R}' \leq \R$ be the subalgebra of $\A^n$ generated by $R$ using $\mu$. We prove by induction on $i\leq n$ that $\pi_{1,\ldots,i}(\R') = \pi_{1,\ldots,i}(\R)$. Suppose that $t \in \R$.

For $i=1$, for any $1 \leq i \leq n$, the tuple $(i, \pi_i(t), \pi_i(t))$ is witnessed by the pair $(t, t)$. Thus, for any representation $R$ of $\R$ and any $1 \leq i \leq n$, we must have $\pi_i(R) = \pi_i(\R)$, which proves the statement.

For $i>1$, by the induction hypothesis, there exists some $t' \in \R'$ such that $\pi_{1,\ldots,i-1}(t) = \pi_{1,\ldots,i-1}(t')$. Let $a = \pi_i(t')$ and $b = \pi_i(t)$. Since $\R' \subseteq \R$, we have that $(i, a, b) \in \operatorname{Sig}(\R) = \operatorname{Sig}(R)$, as it is witnessed by $(t, t') \in \R$. Thus, there must exist a pair $(t_a, t_b) \in R$ witnessing the tuple $(i, a, b)$. Define $t'' \in \R'$ by
$$
t'' = \mu(t', t_a, t_b).
$$
Then, from $\pi_{1,\ldots,i-1}(t_a) = \pi_{1,\ldots,i-1}(t_b)$ and the fact that $\mu$ is a Mal'tsev term, we get
$$
\pi_{1,\ldots,i-1}(t'') = \pi_{1,\ldots,i-1}(\mu(t', t_a, t_b)) = \pi_{1,\ldots,i-1}(t) = \pi_{1,\ldots,i-1}(t').
$$
Additionally, from $\pi_i(t') = \pi_i(t_a) = a$ we have
$$
\pi_i(t'') = \mu(a, a, b) = b = \pi_i(t),
$$
so $\pi_{1,\ldots,i}(t'') = \pi_{1,\ldots,i}(t)$.
\end{proof}

A $k$-ary operation $\varphi: A^k \to A$ on a finite set $A$ is called a \emph{near-unanimity} operation if it satisfies $\varphi(x, y, \ldots, y) = \varphi(y, x, \ldots, y) = \ldots = \varphi(y, y, \ldots, x) = y$ for all $x, y \in A$. Generalized majority-minority (GMM) operations include and extend both near-unanimity and Mal'tsev operations. Intuitively, an operation is called GMM if, when restricted to any two-element subset of its domain, it behaves either like a near-unanimity operation or like a Mal'tsev operation.

\begin{df}[Generalized majority-minority term]
Let $k \geq 3$. A $k$-ary operation $\varphi: A^k \to A$ on a finite set $A$ is called \emph{generalized majority-minority} (GMM) if, for all $a,b \in A$,
\begin{gather*}
    \varphi(x, y, \ldots, y) = \varphi(y, x, \ldots, y) = \ldots = \varphi(y, y, \ldots, x) = y \quad \text{for all } x, y \in \{a, b\}, \\
    \text{or} \\
    \varphi(x, y, \ldots, y) = \varphi(y, y, \ldots, x) = x \quad \text{for all } x, y \in \{a, b\}.
\end{gather*}
In the first case, we say that $a, b\in A$ is a \emph{majority pair} for $\varphi$, and in the second case $a, b\in A$ is a \emph{minority pair}.
\end{df}

\begin{df}[Signature relative to a GMM operation]
Let $\varphi$ be a $(k+1)$-ary GMM operation on $A$. We define the \emph{signature} $\operatorname{Sig_{GMM}}(R)$ of a relation $R\subseteq A^n$ (not
necessarily invariant under $\varphi$) \emph{relative to $\varphi$} to be the set of all tuples $(i, a, b)$ such that $\{a, b\}$ is a minority pair, witnessed by elements in $R$, i.e.,
\begin{gather*}
    \operatorname{Sig}(R) = \{(i, a, b) \in [n] \times A^2 :\ \{a, b\} \text{ is a minority pair,} \\\text{and } \exists\, t_a, t_b \in R \text{ such that } (t_a, t_b) \text{ witnesses } (i, a, b)\}.
\end{gather*}
\end{df}

\begin{df}[Representation relative to a GMM operation]
Let $\varphi$ be a $(k+1)$-ary GMM operation on $A$. A subset $R' \subseteq R$ is called a \emph{representation} of a relation $R\subseteq A^n$ (not
necessarily invariant under $\varphi$) \emph{relative to} $\varphi$ if $\operatorname{Sig_{GMM}}(R) = \operatorname{Sig_{GMM}}(R')$ and $\pi_I(R) = \pi_I(R')$ for every $I \subseteq \{1, \ldots, n\}$ with $|I| \leq k$. If, in addition,
\[
|R'| \leq 2|\operatorname{Sig}(R)| + \sum_{|I| \leq k} |\pi_I(R)|,
\]
then $R'$ is called a \emph{compact representation} of $R$.
\end{df}

Note that if $R'$ is a minimal representation, then
$$|R'| \leq 2n \cdot |A|^2 + \sum_{i=1}^k( \binom{n}{i} \cdot |A|^i).$$

\begin{ex}[\cite{lmcs:2237}]\label{++++____((JHJHJH} Let $R$ be $A^n$, and fix some $d\in A$. Let $\varphi$ be a $(k+1)$-ary GMM operation. We will construct a representation $R'$ of $R$ relative to $\varphi$. Initially $R'$ is empty. First, observe that $\operatorname{Sig_{GMM}}(R)$ contains all $(i,a,b)$ in $[n]\times A^2$ where $\{a,b\}$ is a minority pair. For each tuple $(i,a,b)$ in $\operatorname{Sig_{GMM}}(R)$ we add to $R'$ two elements $e_{i,a}, e_{i,b}$ from Example \ref{alksdj;asjfegh} since $(e_{i,a}, e_{i,b})$ witnesses this tuple. Hence, $\operatorname{Sig_{GMM}}(R')=\operatorname{Sig_{GMM}}(R)$. In a second step we add for each $i_1,\ldots,i_j$ with $j\leq k$ and $i_1<i_2<\ldots <i_j$, and every $a_1,\ldots,a_j\in A$ the element $e^{i_1\ldots i_j}_{a_1\ldots a_j}$ that has $a_l$ on its $i_l$th coordinate for each $l\in \{1,\ldots,j\}$ and $d$ elsewhere. 

It is easy to verify that we obtain
a relation $R'$ such that for all $I\subseteq \{1,\ldots,n\}$ with $|I|\leq k$, $\pi_IR' = A^{|I|} = \pi_IR$, and hence $R'$ is a compact representation of $R$. 
    
\end{ex}

\begin{thm}[\cite{lmcs:2237}]\label{(()()DJKGSKGLdlkjh}
Let $A$ be a finite set, $\A$ be an algebra over $A$ with a $(k+1)$-ary GMM term $\varphi$ and $\R \leq \A^n$. Let $R\subseteq \R$ be a representation of $\R$ relative to $\varphi$. Then $\R$ is generated by $R$ using only $\varphi$.
\end{thm}
The proof of Theorem~\ref{(()()DJKGSKGLdlkjh}, presented in~\cite{lmcs:2237}, is very similar to the proof of Theorem \ref{alalksdjasd}. It proceeds by induction on $k \leq i \leq n$ and, based on the properties of a GMM operation, performs a case distinction for each pair $a, b \in A$ appearing in the projection onto the $i$-th coordinate - depending on whether the pair forms a majority or a minority pair. For details, we refer the reader to the original paper. 

\begin{rem}\label{alkdshfg;asgh}
  We call a $k$-ary operation $f$ on $A$ \emph{idempotent}, if for all $a\in A$, $f(a,a,\ldots,a) = a$. Note that both Mal'tsev and GMM operation are idempotent. 
\end{rem}

\subsection{Constraint satisfaction problems}\label{as;ldhkgf;}

In this section, we present two equivalent definitions of the Constraint Satisfaction Problems (CSPs), define Mal'tsev and GMM CSPs and comment on their complexity. Several definitions, examples, and results have been adapted from \cite{barto_et_al}, \cite{doi:10.1137/050628957} and \cite{https://doi.org/10.48550/arxiv.2210.07383}.

\begin{df}[Constraint satisfaction problem]
     An \emph{instance of a constraint satisfaction problem (CSP)} \color{black} is defined as a triple $\mathcal{P} = (X,A,C)$, where
\begin{itemize}
    \item $X = \{x_1,\ldots,x_{n}\}$ is a finite set of variables,
    \item $A$ is a finite domain of values,
    \item $C$ is a finite set of constraints $C=\{C_1,\ldots,C_m\}$, where each constraint $C_i$ is a pair $C_i = (x^i, R_i)$, with:
\begin{itemize}
\item $x^i = (x_1, \ldots, x_s)$ an $s$-tuple of distinct variables, called the \emph{scope} of $C_i$, and
\item $R_i$ an $s$-ary relation on $A$, called the \emph{constraint relation} of $C_i$.
\end{itemize} 
\end{itemize}
The decision problem for CSP asks whether there exists a \emph{solution} to $\mathcal{P}$, that is, an assignment $f:X\rightarrow A$ such that, for each constraint $C_i=(x^i,R_i)$, the tuple $f(x^i)=(f(x_1),\ldots,f(x_s))$ belongs to $R_i$. If such a solution exists, the instance is called \emph{satisfiable}; otherwise, it is called \emph{unsatisfiable}.
\end{df}

A \emph{constraint language} $\Gamma$ is a set of relations over a finite domain $A$. A constraint satisfaction problem over a constraint language $\Gamma$, denoted by CSP($\Gamma$), is a subclass of the general CSP framework in which every constraint relation used in any instance must be taken from $\Gamma$.

An equivalent formulation of CSP can also be given in terms of homomorphisms between relational structures, with the target structure being fixed.

\begin{df}
    [CSP as a homomorphism problem \cite{BULATOV200531}] $\,$
Let $\mathcal{A}$ be a relational structure over a vocabulary $R_1, \ldots, R_n$. In the \emph{constraint satisfaction problem} associated with $\mathcal{A}$, denoted by CSP($\mathcal{A}$), the question is: given a structure $\mathcal{X}$ over the same vocabulary, does there exist a homomorphism from $\mathcal{X}$ to $\mathcal{A}$? If the answer is positive, we call the instance $\mathcal{X}$ \emph{satisfiable}; otherwise, it is \emph{unsatisfiable}. We refer to $\mathcal{A}$ as the \emph{target structure} and to $\mathcal{X}$ as the \emph{instance (or input) structure}.
\end{df}

The equivalence is straightforward. The translation from the homomorphism form to the constraint form works as follows: interpret the universe $X$ of the structure $\mathcal{X}$ as a set of variables, and treat each tuple $(x_1, \ldots, x_m) \in R^\mathcal{X}$ as a constraint of the form $C = (x_1, \ldots, x_m; R^\mathcal{A})$. Conversely, to move from the constraint form back to the homomorphism form, take the set of variables $X$ as the universe of the instance structure, the domain $A$ as the universe of the target structure, and interpret each constraint $C = (x_1, \ldots, x_m; R)$ as a tuple in a relation $R^\mathcal{X}$ defined over $X$.

It is easy to see that, in its full generality, the CSP is an NP-complete problem. In~\cite{8104069} and~\cite{10.1145/3402029}, it was proved that for every finite constraint language $\Gamma$, CSP($\Gamma$) is either NP-complete or solvable in polynomial time.

\begin{thm}[CSP dichotomy theorem \cite{8104069}, \cite{10.1145/3402029}]
    For each finite constraint language $\Gamma$ over a finite domain $A$, CSP ($\Gamma$) is either polynomial time or NP-complete.    
\end{thm}
This result was made possible by the established correspondence between relational structures and algebras (for the history of the theorem and earlier results, see, for example, \cite{barto_et_al} and \cite{1}). In fact, for a constraint language $\Gamma$, the computational complexity of CSP($\Gamma$) depends only on $\operatorname{Pol}(\Gamma)$. Thus, tractable subclasses of CSPs, as well as the algorithms used to solve them, can be characterized by specific polymorphisms of the corresponding constraint languages. For a constraint language $\Gamma$ over $A$, we say that CSP($\Gamma$) is \emph{Mal'tsev} if there exists a Mal'tsev term on $A$ that preserves every relation in $\Gamma$.

\begin{ex}
    An instance of the $3$-LIN($p$) problem is a system of linear equations $A \cdot \bm{x} = \bm{b}$ over the field $\mathbb{Z}_p$, where each equation involves exactly three variables. If $R$ is the solution space of $A \cdot \bm{x} = \bm{b}$, then the operation $\mu(x, y, z) = x - y + z$ is a polymorphism of $R$. Indeed, suppose that $\bm{x}, \bm{y}, \bm{z} \in R$, then
$$
A \cdot \mu(\bm{x}, \bm{y}, \bm{z}) = A \cdot (\bm{x} - \bm{y} + \bm{z}) = A \cdot \bm{x} - A \cdot \bm{y} + A \cdot \bm{z} = \bm{b} - \bm{b} + \bm{b} = \bm{b}.
$$
\end{ex}

\begin{ex}
    In the general subgroup problem, the domain is a finite group $\mathbb{G}$, and the relations are of the form $gH$, where $H \leq \mathbb{G}^n$ is a subgroup and $g \in \mathbb{G}^n$ (i.e., subgroups and their cosets in $\mathbb{G}^n$). If $U = gH$ is a coset of $\mathbb{G}^n$, then the operation $\mu(x, y, z) = xy^{-1}z$ preserves $U$, since
$$
UU^{-1}U = gH(gH)^{-1}gH = gH H g^{-1} gH = g H H H = gH = U.
$$
\end{ex}

In \cite{articlestdfh} Bulatov and Dalmau proved that Mal'tsev CSPs are solvable in polynomial time by presenting an explicit algorithm, which we call Mal'tsev algorithm. 
\begin{thm}[\cite{articlestdfh}]
    Let $A$ be a finite set and $\mu$ be a Mal’tsev operation on $A$. Then CSP($Inv(\mu)$), where $Inv(\mu)$ is the set of all relations on $A$ invariant under $\mu$, is solvable in polynomial time.
\end{thm}
Note that $\operatorname{Inv}(\mu)$ is an infinite object, while in the classical definition, we consider only finite constraint languages $\Gamma \subseteq \operatorname{Inv}(\mu)$. 

Analogously, we define CSP($\Gamma$) to be GMM, if there exists a GMM term preserving every relation in $\Gamma$. The motivation for combining near-unanimity (NU) algebras and Mal'tsev algebras into a single CSP subclass is that, in both types of algebras, every subalgebra of $\mathbb{A}^n$ has a small generating set. In \cite{lmcs:2237} Dalmau presented a polynomial-time algorithm for GMM CSPs. 

\begin{thm}[\cite{lmcs:2237}]
    Let $A$ be a finite set and $\varphi$ be a GMM operation on $A$. Then CSP($Inv(\varphi)$), where $Inv(\varphi)$ is the set of all relations on $A$ invariant under $\varphi$, is solvable in polynomial time.
\end{thm}

\subsection{The idea of Mal'tsev Algorithm}\label{al;sdkgh;lggggg}

The algorithm developed by Bulatov and Dalmau for solving Mal'tsev CSPs in \cite{doi:10.1137/050628957} based on the fact that any relation invariant under a Mal’tsev term can be generated from its compact representation. In this section, we give the idea of Mal'tsev algorithm. For a detailed description of the algorithm, we refer the reader to the original paper. 

Consider a CSP instance $\mathcal{P} = (X, A, C) = (\{x_1, \ldots, x_n\}, A, \{C_1, \ldots, C_m\})$, where every constraint relation in $C$ is invariant under a Mal'tsev term $\mu$. For each $l \in \{0, \ldots, m\}$, we define the instance $\mathcal{P}_l$ as the CSP instance consisting of the first $l$ constraints of $\mathcal{P}$, that is, $\mathcal{P}_l = (X, A, \{C_1, \ldots, C_l\})$. We denote by $\R_l$ the $n$-ary relation on $A$ defined by:
$$
\R_l = \{(s(x_1), \ldots, s(x_n)) : s \text{ is a solution of } \mathcal{P}_l\}.
$$

In essence, the algorithm constructs a compact representation $R_l$ of $\R_l$ for each $l \in \{0, \ldots, m\}$. In the base case $l = 0$, the instance $\mathcal{P}_0$ contains no constraints, so $\R_0 = \A^n$. A compact representation of $\R_0$ can therefore be constructed as shown in Example~\ref{alksdj;asjfegh}. The algorithm then proceeds iteratively: for each $l$, it computes a compact representation $R_{l+1}$ of $\R_{l+1}$ based on $R_l$ and the next constraint $C_{l+1}$. This step is carried out by invoking the procedure {\sc Next}$(R_l, C_{l+1})$, the appropriate simplification of which is described in Section \ref{+++)_)_)KJ}. 

Mal'tsev algorithm correctly decides whether an instance $\mathcal{P}$ of CSP($\operatorname{Inv}(\mu)$) is satisfiable in time $O(mn^8 + m(n + |S^*|)^4 |S^*| n^2)$, where $n$ is the number of variables in $\mathcal{P}$, $m$ is the number of constraints, and $S^*$ is the largest constraint relation occurring in $\mathcal{P}$.

\subsection{Theory $V^1$, Extended Frege proof system and Translation theorem}\label{==-asl;iyfg;iy}

In this section, we define the two-sorted theory $V^1$ and the extended Frege propositional proof system, and explain the correspondence between these two notions. Most definitions and results are adapted from~\cite{10.5555/1734064},~\cite{krajicek_1995}, and~\cite{krajicek2019proof}. 

Theories of bounded arithmetic known as \emph{second-order} (or \emph{two-sorted} first-order) theories are based on a structured two-sorted framework. They distinguish between two types of variables: variables such as $x, y, z, \ldots$ are \emph{number variables}, representing natural numbers, while variables such as $X, Y, Z, \ldots$ are \emph{set (or string) variables}, representing finite subsets of natural numbers - that is, binary strings. Function and predicate symbols may involve both kinds of variables. Functions are categorized as either \emph{number-valued} or \emph{set-valued}, depending on the type of their output. Similarly, quantifiers are divided into two kinds: \emph{number quantifiers}, which range over number variables, and \emph{string quantifiers}, which range over set variables. The language of these second-order theories extends the usual language of Peano arithmetic $\mathcal{L}_{\mathrm{PA}}$, and is denoted by:
$$
\mathcal{L}^2_{\mathrm{PA}} = \{0, 1, +, \cdot, \lvert\,\rvert, =_1, =_2, \leq, \in\}.
$$
In this notation, the symbols $0, 1, +, \cdot, =_1$, and $\leq$ are function and predicate symbols interpreted over number variables. The function $\lvert X \rvert$, referred to as the \emph{length of $X$}, is number-valued and returns the length (or an upper bound) of the binary string $X$. The binary predicate $\in$ expresses set membership between a number and a set variable; we often write $t \in X$ as $X(t)$, interpreting $X(i)$ as the $i$-th bit of the binary string $X$ of length $\lvert X \rvert$. The symbol $=_2$ denotes equality between sets. The basic properties of all symbols in $\mathcal{L}^2_{\mathrm{PA}}$ are captured by a set of axioms known as $2$-BASIC~\cite{10.5555/1734064}. While the axioms for symbols related to number variables capture the standard properties of natural numbers, the axiom subset in $2$-BASIC that considers the second-sorted variables includes the following axioms: 
\begin{enumerate}
\item $X(y)\to y<|X|$.
\item $y+1 =_1 |X|\to X(y)$.
\item $(|X|=_1|Y|\wedge \forall i<|X|(X(i)\leftrightarrow Y(i)))\iff X=_2Y$.
\end{enumerate}
When the variables $x$ and $X$ do not appear in a term $t$, the expression $\exists x \leq t\, \phi$ abbreviates $\exists x\, (x \leq t \wedge \phi)$, and $\forall x \leq t\, \phi$ abbreviates $\forall x\, (x \leq t \rightarrow \phi)$. Similarly, $\exists X \leq t\, \phi$ stands for $\exists X\, (\lvert X \rvert \leq t \wedge \phi)$, and $\forall X \leq t\, \phi$ stands for $\forall X\, (\lvert X \rvert \leq t \rightarrow \phi)$. Quantifiers written in this way are referred to as \emph{bounded} quantifiers. A formula in which all quantifiers are bounded is called a \emph{bounded formula}.

\begin{df}[Two-Sorted Definability] Let $T$ be a second-sorted theory over vocabulary $\mathcal{L}\supseteq\mathcal{L}^2_{\mathrm{PA}}$, and let $\Phi$ be a set of $\mathcal{L}$-formulas. 
\begin{enumerate}
    \item We say that a predicate symbol $P(\bar x, \bar X)$ not in $\mathcal{L}$ is $\Phi$-\emph{representable} in $T$ if there is an $\mathcal{L}$-formula $\varphi(\bar x,\bar X)$ in $\Phi$ such that
    \begin{gather}\label{ajajaaa}
       P(\bar x, \bar X) \iff \varphi(\bar x,\bar X).
    \end{gather}

    \item We say that a number function $f(\bar x, \bar X)$ not in $\mathcal{L}$ is $\Phi$-\emph{definable} in $T$ if there is an $\mathcal{L}$-formula $\varphi(\bar x,\bar X,y)$ in $\Phi$ such that
\begin{gather}
    T\vdash \forall \bar x\forall \bar X\exists !y\,\,\varphi(\bar x,\bar X, y),
\end{gather}
    meaning that for all $\bar x,\bar X$ theory $T$ proves that there is a unique $y$ such that $\varphi(\bar x,\bar X, y)$, and  
\begin{gather}\label{alalksldjlkh}
    y = f(\bar x, \bar X)\iff \varphi(\bar x,\bar X, y).
\end{gather} 

\item We say that a string function symbol $F(\bar x, \bar X)$ not in $\mathcal{L}$ is $\Phi$-\emph{definable} in $T$ if there is a an $\mathcal{L}$-formula $\varphi(\bar x,\bar X,Y)$ in $\Phi$ such that
\begin{gather}\label{erhgwhrwtjrtj}
    T\vdash \forall \bar x\forall \bar X\exists !Y\,\,\varphi(\bar x,\bar X, y),
\end{gather}
and  
\begin{gather}\label{alalksldjlkh}
    Y = F(\bar x, \bar X)\iff \varphi(\bar x,\bar X, Y).
\end{gather} 
\item  We say that a string function symbol $F(\bar x, \bar X)$ not in $\mathcal{L}$ is $\Phi$-\emph{bit-definable} from $\mathcal{L}$ if there is a formula $\varphi(i,\bar x, \bar X)$ in $\Phi$ and an $\mathcal{L}^2_{\mathrm{PA}}$-number term $t(\bar x, \bar X)$ such that the bit graph of $F$ satisfies 
\begin{equation}\label{s;kajhsd;fkhj}
    F(\bar x, \bar X)(i)\iff (i<t(\bar x, \bar X)\wedge \varphi(i,\bar x, \bar X)).
\end{equation}
\end{enumerate}
We say that (\ref{ajajaaa}) is a \emph{defining axiom} for $P(\bar x, \bar X)$, (\ref{alalksldjlkh}) is a \emph{defining axiom}
for $f(\bar x, \bar X)$ and (\ref{erhgwhrwtjrtj}) is a \emph{defining axiom}
for $F(\bar x, \bar X)$. We say that $f$ and $F$ are \emph{definable} in $T$ if it is $\Phi$-definable in $T$ for some $\Phi$. We say that the formula on the RHS of (\ref{s;kajhsd;fkhj}) is a \emph{bit-defining axiom}, or \emph{bit
definition}, of $F$. Note also that such function $F$ is polynomially bounded
in $T$, meaning that $|F(\bar x, \bar X)|\leq t(\bar x, \bar X)$, and $t$ is a bounding term for $F$.
    
\end{df}

\begin{df}[Number induction axiom scheme]
    Let $\Phi$ be a set of two-sorted formulas, and $\phi\in \Phi$. The \emph{number induction axiom scheme} for the set $\Phi$, denoted by $\Phi$-IND, is the set of formulas
\begin{equation}
  \phi(\bar{x},\bar{X},0) \wedge \forall y (\phi(\bar{x},\bar{X},y)\to \phi(\bar{x},\bar{X},y+1))\to \forall z \phi(\bar{x},\bar{X},z).
\end{equation}
\end{df}

The next axiom scheme allows one to define sets with formulas. 

\begin{df}[Comprehension axiom scheme]
    Let $\Phi$ be a set of two-sorted formulas, and $\phi\in \Phi$. The \emph{comprehension axiom scheme} $\Phi$-COMP is the set of all formulas
\begin{equation}
    \forall y\exists Y\leq y\,\forall z< y \, \phi(\bar{x},\bar{X},z)\iff Y(z).
\end{equation}
\end{df}

The class of formulas in which all number quantifiers are bounded and no string quantifiers appear is denoted by $\Sigma^{1,b}_0 = \Pi^{1,b}_0$. For each $i \geq 0$, the class $\Sigma^{1,b}_{i+1}$ (respectively, $\Pi^{1,b}_{i+1}$) consists of formulas of the form $\exists \bar{X} \leq \bar{t}\, \phi$ (respectively, $\forall \bar{X} \leq \bar{t}\, \phi$), where $\phi$ is a $\Pi^{1,b}_{i}$-formula (respectively, a $\Sigma^{1,b}_{i}$-formula), and $\bar{t}$ is a sequence of $\mathcal{L}^2_{\mathrm{PA}}$-terms that do not contain any variable from $\bar{X}$.

\begin{df}[Theory $V^1$]
Theory $V^1$ is a second-order theory axiomatized by $2$-BASIC, $\Sigma^{1,b}_0$-COMP, and $\Sigma^{1,b}_1$-IND.
\end{df}
Theory $V^1$ corresponds to a polynomial-time reasoning in the following sense.

\begin{thm}
    [\cite{buss1}] $\vspace{1pt}$
    \begin{itemize}
        \item A relation $R$ is in NP iff it is represented by some $\Sigma^{1,b}_1$-formula.
            \item A relation $R$ is in P iff $V^1$ proves that it is both $\Sigma^{1,b}_1$ and $\Pi^{1,b}_1$-definable.
        \item A function $f$ is in FP iff it is $\Sigma^{1,b}_1$-definable in $V^1$.   
    \end{itemize}
\end{thm}

\begin{lem}(Extension by Bit Definition \cite{10.5555/1734064})
Any polynomially bounded number function whose graph is $\Sigma^{1,b}_0$-representable, or a string function whose bit graph is $\Sigma^{1,b}_0$-bit-definable, is $\Sigma^{1,b}_0$-definable in $V^1$. 
\end{lem}

A \emph{propositional proof system} is any polynomial-time function $P$ whose range is exactly the set of all tautologies $TAUT$. For a tautology $\tau\in TAUT$, a string $w$ such that $P(w)=\tau$ is called a $P$-proof of $\tau$. Proof systems are typically specified by a finite set of inference rules of a fixed form, and proofs are constructed by applying these rules sequentially. The complexity of a proof is evaluated in terms of its size and the number of inference steps involved.

An $l$-ary \emph{Frege rule} is any $(l+1)$-tuple of propositional formulas $\varphi_0, \ldots, \varphi_l$ written as
\begin{gather*}
    \frac{\varphi_1, \ldots, \varphi_l}{\varphi_0},
\end{gather*}
such that $\varphi_1, \ldots, \varphi_l$ logically imply $\varphi_0$, denoted $\varphi_1, \ldots, \varphi_l \vDash \varphi_0$. A rule with $l = 0$ is called a \emph{Frege axiom scheme}. A well-known example of a Frege rule is modus ponens:
$$
\frac{\varphi, \,\,\, \varphi \to \psi}{\psi}.
$$
If $F$ is a finite set of Frege rules, an \emph{$F$-proof} of a formula $\psi$ from formulas $\varphi_1, \ldots, \varphi_t$ is any sequence of formulas $\theta_1, \ldots, \theta_k$ such that $\theta_k = \psi$, and for all $i = 1, \ldots, k$, either $\theta_i \in \{\varphi_1, \ldots, \varphi_t\}$, or it is inferred from some earlier $\theta_j$'s ($j < i$) by a rule from $F$. We denote by $\varphi_1, \ldots, \varphi_t \vdash_F \psi$ the fact that there exists an $F$-proof of $\psi$ from $\varphi_1, \ldots, \varphi_t$. A propositional \emph{Frege proof system} $F$ is a finite set of Frege rules that is sound and implicationally complete, meaning that for any propositional formulas $\varphi_1, \ldots, \varphi_n, \psi$,
$$
\varphi_1, \ldots, \varphi_n \vDash \psi \iff \varphi_1, \ldots, \varphi_n \vdash_F \psi.
$$

\begin{df}[Extended Frege proof system] 
If $F$ is a Frege proof system, an extended Frege proof is a sequence of formulas $\theta_1, \ldots, \theta_k$ such that each $\theta_i$ is either derived from some earlier $\theta_j$'s by a rule from $F$, or has the form $q \equiv \psi$, where:
\begin{enumerate}
    \item the atom $q$ appears in neither $\psi$ nor in any $\theta_j$ for $j < i$,
    \item the atom $q$ does not appear in $\theta_k$,
    \item $\alpha \equiv \beta$ abbreviates $(\alpha \wedge \beta) \vee (\neg \alpha \wedge \neg \beta)$.
\end{enumerate}
A formula of this form is called an \emph{extension axiom}, and $q$ is called an \emph{extension atom}. An \emph{extended Frege proof system} $EF$ is a proof system in which proofs are extended Frege proofs. 
\end{df}

There is a well-known translation of $\Sigma^{1,b}_0$-formulas, described for example in~\cite{10.5555/1734064}, which transforms any formula $\phi(\bar{x}, \bar{X}) \in \Sigma^{1,b}_0$ into a corresponding family of propositional formulas
\begin{equation}
||\phi(\bar{x},\bar{X})|| = \{\phi(\bar{x},\bar{X})[\bar{m},\bar{n}]: \bar{m},\bar{n} \in \mathbb{N}\}.
\end{equation}
We will not explain the translation in detail, but the basic idea is as follows: for every free set variable $X$ in $\varphi$ of length $n$, we introduce propositional variables $p^X_0, p^X_1, \ldots, p^X_{n - 1}$, where each $p^X_i$ is intended to represent $X(i)$. The corresponding propositional formula is then constructed using inductive rules. 

\begin{lem}[\cite{10.5555/1734064}]\label{LPropositionalTranslation} For every $\Sigma^{1,b}_0(\mathcal{L}^2_{\mathrm{PA}})$ formula $\phi(\bar{x}, \bar{X})$, there exists a constant $d \in \mathbb{N}$ and a polynomial $p(\bar{m}, \bar{n})$ such that for all $\bar{m}, \bar{n} \in \mathbb{N}$, the propositional formula $\phi(\bar{x}, \bar{X})[\bar{m}, \bar{n}]$ has depth at most $d$ and size at most $p(\bar{m}, \bar{n})$.
\end{lem}
The correspondence between theory $V^1$ and the extended Frege proof system is captured by the following result. 

\begin{thm}[$V^1$ Translation \cite{10.5555/1734064}] \label{Translation} Suppose that $\varphi(\bar{x},\bar{X})$ is a $\Sigma^{1,b}_0$-formula such that $
V^1 \vdash \forall\bar{x}\forall\bar{X}\varphi(\bar{x},\bar{X})
$.
Then the family of propositional formulas $||\varphi(\vec{x},\vec{X})||= \{\varphi(\vec{x},\vec{X})[\vec{m},\vec{n}]: \vec{m},\vec{n} \in \mathbb{N}\}$ has polynomial size extended Frege proofs and these proofs can be constructed by a $p$-time algorithm.
\end{thm}

\subsection{Auxiliary functions, relations, axioms and notation in theory $V^1$}\label{asdkjfg;lkg}

In this section, we introduce several general auxiliary functions and relations that are definable in theory $V^1$ and will be used in the formalization of Mal'tsev algorithm. Most definitions are adopted from \cite{10.5555/1734064}. 

\begin{notation}
    We will generally use $i, j, k, \ldots$ or $i_1, i_2, \ldots$ for number variables representing indices or array dimensions, and $a, b, c, \ldots$, $x, y, z, \ldots$, or $a_1, a_2, \ldots$, $x_1, x_2, \ldots$ for number variables representing elements of sets.
\end{notation}

For any $x,y \in \mathbb{N}$, we define \emph{pairing function} $\langle x,y\rangle$ as the following term.
\begin{equation}
\langle x,y\rangle = \frac{(x + y)(x + y + 1)}{2} + y.
\end{equation}
It is easy to show that $V^1$ proves the following properties of the pairing function:
\begin{enumerate}
 \item $\forall x_1\forall x_2\forall y_1\forall y_2\,\, (\langle x_1,y_1\rangle = \langle x_2,y_2\rangle \to x_1=x_2 \wedge y_1=y_2)$,
\item $\forall z \exists x\exists y\,\, (\langle x,y\rangle = z)$,
\item $\forall x\forall y\,\,$  $(x,y \leq \langle x,y\rangle < (x+y+1)^2)$.
\end{enumerate}
The pairing function can be iterated to encode tuples of arbitrary length. For any $k > 2$, $k$-tuples can be defined inductively by setting
\begin{equation}
\langle x_1,x_2,\ldots,x_k\rangle = \langle \ldots \langle \langle x_1, x_2\rangle, x_3 \rangle ,\ldots, x_k \rangle.
\end{equation}
We refer to the term $\langle x_1, x_2, \ldots, x_k \rangle$ as the \emph{tupling function}. Note that for any $x_i \in \{x_1, x_2, \ldots, x_k\}$,
\begin{equation}\label{alksjf}
\begin{gathered}
x_i \leq \langle x_1, x_2, \ldots, x_k \rangle < (x_1 + x_2 + \ldots + x_k + 1)^{2^k}.
\end{gathered}
\end{equation}
For any constant $m\geq 2$, we can use the tupling function to code a $k$-dimensional bit array by a single string $H$ by $ H(\langle x_1,\ldots,x_m\rangle)$; we will write $ H(\langle x_1,\ldots,x_m\rangle)$ as $H(x_1,\ldots,x_m)$.

We use the tupling function to encode functions as sets. To express that a set $H$ represents a function (i.e., a well-defined map) from sets $X_1, \ldots, X_n$ to a set $Y$, we assert that for all $x_1 \in X_1, \ldots, x_n \in X_n$, there exists a unique $y \in Y$ such that $H(x_1, \ldots, x_n, y)$ holds. For example, we say that a set $H$ is a well-defined binary map from sets $A$ and $B$ to $C$ if it satisfies the following relation:
\begin{equation}
\begin{gathered}
wMap_2(A, B, C, H) \iff \forall a \in A\, \forall b \in B\, \exists c \in C\, H(a, b, c)\, \wedge \\
\forall a \in A\, \forall b \in B\, \forall c_1 \in C\, \forall c_2 \in C\, (H(a, b, c_1) \wedge H(a, b, c_2) \rightarrow c_1 = c_2).
\end{gathered}
\end{equation}
We will then define the following two index functions:
\begin{enumerate}
\item The string function $row_1(i, Z)$, also written as $Z^{[i]}$, represents the $i$-th row of a binary array $Z$ and is defined by the following bit-definition:
\begin{equation}
\begin{gathered}
Z^{[i]}(a) = row_1(i, Z)(a) \iff (a<|Z|\wedge Z(i, a)).
\end{gathered}
\end{equation}

We use $row_1$ to encode a sequence of strings $Z^{[1]}, \ldots, Z^{[k]}$ as a single string $Z$. More generally, we can represent $m$-ary sub-arrays using the tupling function. For instance, to represent the family of sets $Z^{[i_1 \ldots i_n]}$, we define a similar function $row_n(i_1, \ldots, i_n, Z)$ such that:
\begin{equation}
\begin{gathered}
Z^{[i_1 \ldots i_n]}(a_1, \ldots, a_m) = row_n(i_1, \ldots, i_n, Z)(a_1, \ldots, a_m) \iff \\
\langle a_1, \ldots, a_m \rangle < |Z| \wedge Z(i_1, \ldots, i_n, a_1, \ldots, a_m).
\end{gathered}
\end{equation}

Using similar explicit definitions, we can mix any subsets of indices in a $k$-ary array to extract sub-arrays as needed.

\item Let $a$ be the smallest element of $Z^{[i]}$, or let $a = |Z|$ if $Z^{[i]}$ is empty. The number function $seq(i, Z)$ (also denoted by $Z^{(i)}$) is defined by the following axiom:
\begin{gather*}
a = seq(i, Z) \iff \left(a<|Z| \wedge Z(i, a) \wedge \forall b < a\, \neg Z(i, b) \right) \vee \\
\hspace{0pt} \vee \left( \forall b < |Z|\, \neg Z(i, b) \wedge a = |Z| \right).
\end{gather*}

We use the function $seq$ to allow $Z$ to encode a sequence $a_0, a_1, \ldots$ of numbers. In a similar way, we can define sequences of elements in $n$-ary arrays. For example, when $Z(i_1, \ldots, i_n, a)$ holds, we write $Z^{(i_1 \ldots i_n)}$ for the value $a$ such that $seq(i_1, \ldots, i_n, Z) = a$. Note that by explicitly combining coordinates, we can extract the smallest element corresponding to any subset of indices.
\end{enumerate}

\begin{notation}
We will always use square or round brackets in expressions like $Z^{[i]}$ or $Z^{(i)}$ to indicate a sub-dimension or an element of an array, in order to distinguish them from regular indices or indices used in abbreviations or other annotations. For sets representing functions, such as $H(x_1, \ldots, x_n, y)$, we will typically write $H(x_1, \ldots, x_n) = y$ - rather than $H^{(x_1, \ldots, x_n)} = y$ - to reflect the uniqueness of $y$.
\end{notation}
To define the maximum and minimum elements of a set $Z$, we introduce the number functions $max$ and $min$ as follows:
\begin{equation}\label{ala9s7d5tdhgdfr}
\begin{gathered}
\hspace{0pt} max(Z) = |Z| - 1, \\
min(Z) = x \iff \forall y < |Z|\, (Z(y) \rightarrow x \leq y).
\end{gathered}
\end{equation}
Given a set $X$, the \emph{census function} $\#X(n)$ is a number-valued function defined for all $n \leq |X|$, where $\#X(n)$ denotes the number of elements $x < n$ such that $x \in X$. In particular, $\#X(|X|)$ gives the total number of elements in $X$. The following relation specifies that $\#X$ is the census function for $X$:
\begin{equation}
\begin{gathered}
Census(X, \#X) \iff \#X \leq \langle |X|, |X| \rangle \wedge \#X(0) = 0 \wedge \forall x < |X|\, \big( \\
x \in X \rightarrow \#X(x+1) = \#X(x) + 1 \wedge 
x \notin X \rightarrow \#X(x+1) = \#X(x) \big).
\end{gathered}
\end{equation}
$V^1$ proves the \emph{counting axiom}, i.e. that for every set $X$ there exists such function $\#X$. Given any set $X$, consider $\Sigma^{1,b}_1$-induction on $n \leq |X|$ for the formula
\begin{equation}
\begin{gathered}
\varphi(n, X) = \exists \#X \leq \langle n, n \rangle\, \big( \#X(0) = 0 \wedge \forall\, 0 \leq x < n \\
(x \in X \rightarrow \#X(x+1) = \#X(x) + 1) \wedge 
(x \notin X \rightarrow \#X(x+1) = \#X(x)) \big).
\end{gathered}
\end{equation}

\section{Soundness of Mal'tsev algorithm in theory $V^1$}\label{alskdyf;kajdsghy;kgj}

The goal of this section is to establish an upper bound on the proof complexity of CSPs admitting a Mal'tsev term, in the sense that a family of propositional formulas expressing the unsatisfiability of CSP instances admit short proofs in the extended Frege proof system. To achieve this, we first prove the corresponding universal unsatisfiability statement in theory $V^1$, and then apply the translation theorem.

To prove the universal statement in $V^1$, it suffices to show that $V^1$ proves the soundness of Mal'tsev algorithm. Recall the idea of Mal'tsev algorithm from Section \ref{al;sdkgh;lggggg}. Assume that $V^1$ proves the following: for every $0 < l \leq m$, if $R_{l-1}$ is a compact representation of $\R_{l-1}$, then {\sc Next}$(R_{l-1}, C_l)$ outputs a compact representation of $\R_l$. Then the statement
\begin{gather*}
    R_0 = \{e_{i,a} : (i,a) \in [n] \times D_i\} \wedge R_m = \emptyset \wedge \\
    \wedge \forall\, 0 < l \leq m\,\, R_l := \text{{\sc Next}}(R_{l-1}, C_l)
\end{gather*}
constitutes a proof of unsatisfiability. 

Indeed, consider a CSP instance $\mathcal{P} = (X, A, \{C_1, \ldots, C_m\})$. If for all $0 < l \leq m$, the set $R_l$ is a compact representation of $\R_l$, the initial set $R_0$ is chosen as in Example~\ref{alksdj;asjfegh}, and the compact representation $R_m$ of the solution set to the original instance is empty, then the instance has no solution. Importantly, to establish unsatisfiability, we do not need to prove that each $R_l$ generates the full solution set $\R_l$. It is sufficient to show that the signature of $R_l$ coincides with the signature of $\R_l$ and $R_l\subseteq \R_l$. Then if $R_m$ is empty, the signature of the solution set $\R_m$ is also empty, which implies that the solution set itself is empty. This is a significant advantage, as it allows us to avoid replicating the full proof of Theorem~\ref{alalksdjasd}, which involves exponentially large sets.

To prove the soundness of Mal'tsev algorithm, we must first show that the procedure {\sc Next} can be formalized using the induction available in $V^1$ - that is, the sequence of sets $R_0, \ldots, R_m$ is constructible within $V^1$. Next, we need to show that $V^1$ proves the correctness of the {\sc Next} procedure. Specifically, for every $0 < l \leq m$, $V^1$ should prove that if $R_{l-1} \subseteq \R_{l-1}$ and $\operatorname{Sig}(R_{l-1}) = \operatorname{Sig}(\R_{l-1})$, then {\sc Next}$(R_{l-1}, C_l)$ produces a set $R_l \subseteq \R_l$ such that $\operatorname{Sig}(R_l) = \operatorname{Sig}(\R_l)$.

\subsection{Formalization of a CSP instance}\label{alskdj;lkys}

To simplify the formalization for the reader, we restrict our attention in this paper to constraint languages in which all constraints are at most binary. There are two justifications for this restriction. The first point is that the formalization presented in this paper can be easily extended to any finite constraint language $\Gamma$, since the maximum arity of the constraint relations can be fixed in advance. We will elaborate on this point in more detail during the formalization. 

The second point concerns classical log-space reductions between constraint languages, which are commonly referred to as \emph{$pp$-constructibility}. While we do not introduce this notion in detail here, an extensive treatment can be found in~\cite{barto_et_al}. These reductions preserve both computational complexity and proof complexity of CSP problems; see~\cite{10.1145/3265985}. It is known that for two constraint languages $\Gamma$ and $\Gamma'$, the language $\Gamma$ $pp$-constructs $\Gamma'$ if and only if $Pol(\Gamma')$ contains operations satisfying all height-$1$ identities satisfied by operations in $Pol(\Gamma)$. Since a Mal'tsev term is idempotent, Mal'tsev identities can be written in height-$1$ form: $\mu(x,y,y) = \mu(y,y,x) = \mu(x,x,x)$. 
\begin{thm}[\cite{barto_et_al}]
    For any constraint language $\Gamma$, there exists an idempotent constraint language $\Gamma'$ such that:
    \begin{enumerate}
        \item all relations in $\Gamma'$ are of arity at most two, and
        \item $\Gamma$ and $\Gamma'$ $pp$-construct each other.
    \end{enumerate}
\end{thm}
This implies that for any constraint language $\Gamma$ preserved by a Mal'tsev term, one can construct an idempotent constraint language $\Gamma'$ with only binary relations such that there exists an operation $\mu'\in Pol(\Gamma')$ satisfying $\mu'(x,y,y) = \mu'(y,y,x) = \mu'(x,x,x) = x$, which is a Mal'tsev term.

\begin{notation}
To simplify notation, we use capital letters to denote relations on sets and numbers, and lowercase letters to denote functions. When arguments of a relation are clear from context and do not affect the meaning, we may omit them. Elements of sets are indexed starting from $0$, while indices not referring to elements of sets (e.g., indices of a sequence of relations) start from $1$.
\end{notation}

\begin{notation}
We denote by $[s]$ the set $X$ of size $s$ such that $X(i)$ holds for every $i < s$.
\end{notation}

Consider any constraint language $\Gamma$ with at most binary relations. For an instance of CSP($\Gamma$) with $n$ variables $X=\{x_1,\ldots,x_n\}$, every constraint is either of the form $C = (x_i, D)$, where $D$ is a unary relation in $\Gamma$, or $C = (x_i, x_j, E)$, where $E$ is a binary relation in $\Gamma$. We assume that for each variable $x_i \in X$ there is exactly one unary relation, and for each ordered pair $(x_i, x_j)$ there is at most one binary relation. (In cases with multiple constraints for the same variables and pairs of variables, we can reduce them by computing intersections.) Each unary relation can be interpreted as the domain of the variable $x_i$, and each binary constraint as a directed edge between the variables $(x_i, x_j)$. This representation allows us to treat any CSP instance as a homomorphism problem between relational structures that are specific types of directed graphs. 

\begin{df}\label{DIGRAPHSS}
A \emph{directed input graph} of size $n$ is a pair $\mathcal{X} = ([n], E_\mathcal{X})$, where $E_\mathcal{X}(i,j)$ is a binary relation on the vertex set $[n]$ indicating the presence of an edge from $i$ to $j$. A \emph{target digraph with domains} is a tuple of sets $\ddot{\mathcal{A}} = ([q], V_{\ddot{\mathcal{A}}}, E_{\ddot{\mathcal{A}}})$, where:
\begin{itemize}
    \item $V_{\ddot{\mathcal{A}}} < \langle n, q \rangle$ is the superdomain. The predicate $V_{\ddot{\mathcal{A}}}(i, a)$ indicates that the domain for variable $x_i$ contains the element $a$. For brevity, we write $D_i := V_{\ddot{\mathcal{A}}}^{[i]}$.
    
    \item $E_{\ddot{\mathcal{A}}} < \langle n, n, q, q \rangle$ encodes the allowed edges between domains. That is, $E_{\ddot{\mathcal{A}}}^{[ij]}(a,b)$ holds only if
    \begin{equation}
        \begin{gathered}
            E_{\ddot{\mathcal{A}}}^{[ij]}(a,b) \rightarrow D_i(a) \wedge D_j(b).
        \end{gathered}
    \end{equation}
\end{itemize}
We denote by the predicate DG$(\Theta)$ that a pair of sets $\Theta = (\mathcal{X}, \ddot{\mathcal{A}})$ satisfies all the above conditions. This representation allows us to model homomorphisms from $\mathcal{X}$ to $\ddot{\mathcal{A}}$, using edge relations $E_{\ddot{\mathcal{A}}}^{[ij]}$ and domain sets $D_i$ for all vertices $x_1, \ldots, x_n$.
\end{df}

Let us fix $b_h := \langle n, q \rangle$ as the upper bound for sets encoding homomorphisms.

\begin{df}[Homomorphism from a digraph $\mathcal{X}$ to a digraph with domains $\ddot{\mathcal{A}}$]\label{krjfhpakrhfgdrg}
A well-defined map $H < b_h$ from $[n]$ to $[q]$ is called a \emph{homomorphism between the input digraph} $\mathcal{X} = ([n], E_{\mathcal{X}})$ of size $n$ and the \emph{target digraph with domains} $\ddot{\mathcal{A}} = ([q], V_{\ddot{\mathcal{A}}}, E_{\ddot{\mathcal{A}}})$, where $V_{\ddot{\mathcal{A}}} < \langle n, q \rangle$, if $H$ maps each $i \in V_{\mathcal{X}}$ to its corresponding domain $D_i$ and preserves edges.
\begin{equation}\label{HOMINSTAG}
\begin{gathered}
Hom(\mathcal{X}, \ddot{\mathcal{A}}, H) \iff  
wMap_1([n], [q], H)\,\wedge (\forall i < n\, \forall a < q\, (H(i) = a \rightarrow D_i(a)))\,\wedge \\
\forall i, j < n\, \forall a, b < q\, 
(E_{\mathcal{X}}(i, j) \wedge H(i) = a \wedge H(j) = b \rightarrow E^{[ij]}_{\ddot{\mathcal{A}}}(a, b))
\end{gathered}
\end{equation}
The existence of such $H$ can be expressed by the $\Sigma^{1,b}_{1}$-formula $\exists H < b_h\,Hom(\mathcal{X}, \ddot{\mathcal{A}}, H)$. We will sometimes abbreviate the expression $\exists H < b_h,Hom(\mathcal{X}, \ddot{\mathcal{A}}, H)$ as $Hom(\mathcal{X}, \ddot{\mathcal{A}})$.
\end{df}

\begin{rem}
    Consider any finite constraint language $\Gamma$, and let $s$ be the maximum arity of relations in $\Gamma$. Then, any instance of CSP($\Gamma$) can be viewed as a homomorphism problem between an input hypergraph $\mathcal{X}$ and a target hypergraph with domains $\ddot{\mathcal{A}}$, where the cardinality of each edge is at most $s$. Specifically, $\mathcal{X}$ can be represented as a tuple of sets $([n], E^2_{\mathcal{X}},\ldots,E^s_{\mathcal{X}})$, and $\ddot{\mathcal{A}}$ as $([q], V_{\ddot{\mathcal{A}}}, E^2_{\ddot{\mathcal{A}}},\ldots,E^s_{\ddot{\mathcal{A}}})$. The definition of a homomorphism then extends in a straightforward manner to this setting.
\end{rem}

Note that with Definitions \ref{DIGRAPHSS} and \ref{krjfhpakrhfgdrg}, we no longer need to treat unary constraints as explicit constraints - they are now directly encoded in the structure of the target digraph. Therefore, we can focus solely on binary constraints. We denote the number of edges in the relation $E_{\mathcal{X}}$, which can be counted using the census function $\#E_{\mathcal{X}}$, by $m$.

The solution set for the instance $\Theta = (\mathcal{X}, \ddot{\mathcal{A}})$ consists of all homomorphisms from $\mathcal{X}$ to $\ddot{\mathcal{A}}$, which we denote by $\mathscr{R}_{\Theta}$. Since this set can be exponentially large, we define it in $V^1$ only indirectly, using a $\Sigma^{1,b}_0$-formula:
\begin{equation}
H \in \mathscr{R}_{\Theta} \iff Hom(\mathcal{X}, \ddot{\mathcal{A}}, H).
\end{equation}

\subsection{Formalization of universal algebra notions}\label{askfdh;qweyi}

The algorithm applies to any finite algebra that possesses a Mal'tsev term. A fixed algebra $\mathbb{A}$ of size $q$ with a Mal'tsev term $\mu$ can be represented as a pair of sets $([q], M)$, where $M$ encodes the ternary operation $\mu$. Formally, we define:
\begin{equation}
\begin{gathered}
Maltsev([q], M) \iff wMap_3([q],[q],[q],[q], M) \wedge \forall a<q \forall b<q\,\, M(a,a,b) = M(b,a,a) = b.
\end{gathered}
\end{equation}

For any subset $R'\subseteq A^n$, its signature can be encoded by a set $S(i,a,b)$ of size at most $b_s:=\langle n,$ $q,$ $q\rangle< (n+2q+1)^8$. For a compact representation $R$ of $R'$, for any tuple $(i,a,b)$ in a signature, we need to select just two elements $T_a, T_b$ from $R'$ witnessing the tuple, which we can represent as well-defined maps from $[n]$ to $[q]$. In the case where $a=b$, we assume that for the tuple $(i,a,a)$ we consider a pair $(T,T)$ with $T(i)=a$. Thus, we can encode a compact representation $R$ as a set $R(i,a,b,j,c)$, where the first three indices correspond to a tuple $(i,a,b)$ witnessed by maps $T_a, T_b$ from $[n]$ to $[q]$, and the last two indices define the map $T_a$. For the dual tuple $(i,b,a)$, the last two indices of $R(i,b,a,j,c)$ define the map $T_b$. With this encoding, every index $\langle i,a,b\rangle $ corresponds to exactly one map.

\begin{rem}\label{a=++++}
Note that in Definition \ref{askdjfh;lkasf} of the signature, two tuples $(i,a,b)$ and $(i,b,a)$ may correspond to two different pairs $(t_a,t_b)$ and $(t_b',t_a')$. In our case, for both tuples we consider the same pair of maps $(T_a, T_b)$ and $(T_b, T_a)$, so we ensure they are stored under unique indices. In the case where $a=b$, the indices $\langle i,a,b\rangle = \langle i,b,a\rangle$, so we must ensure that the same map is indeed stored under both tuples.
\end{rem}
Then we say that a set $S$ is a signature for a compact representation $R$ if, for every tuple $(i,a,b)$ such that $S(i,a,b)$ holds, the sets $R_{[iab]}$ and $R_{[iba]}$ contain well-defined maps from $[n]$ to $[q]$ that witness this tuple.
\begin{equation}\label{eq1232346}
\begin{gathered}
    Sig_{CR}(S, R) \iff \forall i<n\forall a<q\forall b<q\,\, \big[ S(i,a,b)\longleftrightarrow
R^{[iab]}(i)=a \wedge R^{[iba]}(i)=b \wedge \\
\forall j<i\,\,R^{[iab]}(j) = R^{[iba]}(j)\wedge wMap_1([n], [q], R^{[iab]})\wedge wMap_1([n], [q], R^{[iba]})\big].
\end{gathered}
\end{equation}
For any CSP instance $\Theta=(\mathcal{X},\ddot{\mathcal{A}})$, its solution set $\mathscr{R}_{\Theta}$ consists of all homomorphisms from $\mathcal{X}$ to $\ddot{\mathcal{A}}$. We say that a set $S$ is a signature for the solution set $\mathscr{R}_{\Theta}$ if, for every tuple $(i,a,b)$ such that $S(i,a,b)$ holds, there exist two homomorphisms $T_a, T_b$ from $\mathcal{X}$ to $\ddot{\mathcal{A}}$ that witness this tuple.
\begin{equation}\label{eqj3jkr}
\begin{gathered}
   Sig_{SS}(S, \mathcal{X},\ddot{\mathcal{A}}) \iff \forall i<n\forall a<q\forall b<q\,\, \big[S(i,a,b)\longleftrightarrow
    \exists T_a,T_b\leq b_h\\
T_a(i)=a \wedge T_b(i)=b \wedge \forall j<i\,\, T_a(j) = T_b(j)\wedge \\
\wedge Hom(\mathcal{X},\ddot{\mathcal{A}}, T_a)\wedge Hom(\mathcal{X},\ddot{\mathcal{A}}, T_b)\big].
\end{gathered}
\end{equation}
We say that $R$ is a subset of $\mathscr{R}_{\Theta}$ if
\begin{equation}
    \begin{gathered}
        R\subseteq \mathscr{R}_{\Theta} \iff \forall i<n\forall a<q\forall b<q \,\, Hom(\mathcal{X},\ddot{\mathcal{A}}, R^{[iab]}).
    \end{gathered}
\end{equation}
Finally, we say that $Sig(R) = Sig(\mathscr{R}_{\Theta})$ if
\begin{equation}\label{akajsgdfl7}
    \begin{gathered}
        Sig(R)=Sig(\mathscr{R}_{\Theta}) \iff \exists S< b_s\, Sig_{CR}(S,R)\wedge Sig_{SS}(S,\Theta).
    \end{gathered}
\end{equation}
Further, for any $n$, we introduce a string function $\mu$ that takes any three maps from $[n]$ to $[q]$ and returns a new map $H$, defined by its bit-definition for all $i < n$ and $a < q$:
\begin{equation}\label{alalkdhjhfr}
  \begin{gathered}
     \mu(H_1, H_2, H_3)(i,a)\iff \exists a_1<q\exists a_2<q\exists a_3<q\,M(a_1,a_2,a_3)=a\wedge\\
       \hspace {0pt}H_1(i)=a_1\wedge H_2(i)=a_2\wedge H_3(i)=a_3\wedge\\
       wMap_1([n],[q],H_1) \wedge wMap_1([n],[q],H_2)\wedge wMap_1([n],[q],H_3)
  \end{gathered}
\end{equation}
Note that $\mu$ encodes a well-defined map rather than an arbitrary binary array. It is based on the fixed set $M$ representing the Mal'tsev polymorphism, and 
\begin{equation}
    \begin{gathered}
V^1\vdash\forall i\exists ! a\, \big( \exists a_1<q\exists a_2<q\exists a_3<q\,M(a_1,a_2,a_3)=a\wedge H_1(i)=a_1\wedge H_2(i)=a_2\wedge H_3(i)=a_3\big),
    \end{gathered}
\end{equation}
since all $H_1, H_2, H_3$, and $M$ are well-defined maps.

\subsection{Correspondence between CSP instance and Mal'tsev algebra}\label{lh;alkdhg;lih;}

We say that a set $F$ is a well-defined map of arity $3$ on the set $[q]$ that preserves a $2$-ary relation $R$ on $[q] \times [q]$ if it satisfies the following relation.
\begin{equation}
 \begin{gathered}
 Pol_{3,2}([q], F, R) \iff wMap_3([q],[q],[q],[q],F) \wedge \forall a_1,a_2,a_3<q\forall b_1,b_2,b_3<q, \\
 \hspace{0pt}\forall c_1,c_2<q\,\,R(a_1,b_1)\wedge R(a_2,b_2) \wedge R(a_3,b_3) \wedge\\
 \hspace{0pt}F(a_1,a_2,a_3)=c_1\wedge F(b_1,b_2,b_3)=c_2 \rightarrow R(c_1,c_2).
 \end{gathered}
\end{equation}
In the same fashion, we can define a ternary polymorphism for unary relations. Finally, we introduce the $\Sigma^{1,b}_0$-relation indicating that a CSP instance $\Theta = (\mathcal{X}, \ddot{\mathcal{A}})$ is an instance of a relational structure that corresponds to the algebra $\A = ([q], M)$.
\begin{equation}
    \begin{gathered}
         Inst(\mathcal{X}, [l], V_{\ddot{\mathcal{A}}},E_{\ddot{\mathcal{A}}},[q], M)\iff DG(\mathcal{X},\ddot{\mathcal{A}})\wedge [l]=[q]\wedge\\
\forall i, j<n \,Pol_{3,2}([q],M,E^{[ij]}_{\ddot{\mathcal{A}}})\wedge \forall i<n \,Pol_{3,1}([q],M,D_i).
    \end{gathered}
\end{equation}
The following lemma is very simple to prove.
\begin{lem}\label{simplemm}
$V^1$ proves that for any CSP instance $\Theta$ corresponding to algebra $([q], M)$, and any homomorphisms $H_1, H_2, H_3$ from $\mathcal{X}$ to $\ddot{\mathcal{A}}$, the map $H = \mu(H_1, H_2, H_3)$ is again a homomorphism from $\mathcal{X}$ to $\ddot{\mathcal{A}}$.
\end{lem}
\begin{proof}
Since $M$ encodes a polymorphism, it follows that all domains $D_0, \ldots, D_{n-1}$ are subalgebras of $\A$, and each relation $E_{\mathcal{A}}^{[ij]}$ - the set of edges from $D_i$ to $D_j$ - is a subalgebra of $D_i \times D_j$, because it is compatible with $M$. The proof proceeds by contradiction: suppose there exists an edge $(x_i, x_j) \in E_{\mathcal{X}}$ such that $H$ does not map it to an edge in $\ddot{\mathcal{A}}$. However, since all homomorphisms $H_1$, $H_2$, and $H_3$ map $(x_i, x_j)$ to edges in $E^{[ij]}_{\ddot{\mathcal{A}}}$, this failure can occur only if $M$ does not preserve the relation $E^{[ij]}_{\ddot{\mathcal{A}}}$.
\end{proof}

\subsection{Formalization of Mal'tsev algorithm in $V^1$}\label{alskhf;shgf}

Consider any CSP instance $\Theta=(\mathcal{X},\ddot{\mathcal{A}})$ such that $\Theta$ is an instance of a relational structure corresponding to the algebra $\A=([q], M)$. Then consider the sequence of CSP instances $\Theta_0, \ldots, \Theta_m$ with $\ddot{\mathcal{A}}$ as the target digraph for all $m$,
$\mathcal{X}_0:=([n], E_{\mathcal{X}_0}), \ldots, \mathcal{X}_m:=([n], E_{\mathcal{X}_m}) = \mathcal{X}$, where $E_{\mathcal{X}_0}$ is an empty set, and each $\mathcal{X}_{i-1}$ is constructed from $\mathcal{X}_{i}$ by removing one edge, i.e., for all $0 < l \leq m$:
\begin{equation}
    \begin{gathered}
        \forall i,j<n, \,\, E_{\mathcal{X}_{l-1}}(i,j) \iff E_{\mathcal{X}_{l}}(i,j) \wedge min(E_{\mathcal{X}_l}) < \langle i,j\rangle.
    \end{gathered}
\end{equation}
We remove edges one by one, always selecting the smallest remaining edge. We gather all compact representations $R_0, \ldots, R_m$ of $\mathscr{R}_{\Theta_0}, \ldots, \mathscr{R}_{\Theta_m}$ into a single set $R$ such that, for all $i < n$ and for all $a, b \in D_i$ and $d_j = \min(D_j)$,
\begin{equation}
    \begin{gathered}\label{askjfh;ashfd}
R^{[0]}(i, a,b ,i,a)\wedge \forall j \neq i < n \, R^{[0]}( i, a,b,j,d_j), \\
\forall 0 < l \leq m\,\, R^{[l]} = next(R^{[l-1]}, \min(E_{\mathcal{X}_l}), E^{[\min(E_{\mathcal{X}_l})]}_{\ddot{\mathcal{A}}}),
    \end{gathered}
\end{equation}
where the function $next(R^{[l-1]}, \min(E_{\mathcal{X}_l}), E^{[\min(E_{\mathcal{X}_l})]}_{\ddot{\mathcal{A}}})$ will be defined in the following sections.

\begin{rem}\label{&&^*&}
In this encoding, for every pair of indices $\langle i,a,b\rangle$ and $\langle i,b,a\rangle$, we store two maps $T_a$ and $T_b$ that agree on all coordinates $j < i$, with $T_a(i) = a$ and $T_b(i) = b$. In the case where $a = b$, both indices $\langle i,a,b\rangle$ and $\langle i,b,a\rangle$ correspond to a single map $T$ with $T(i) = a$. We will maintain this convention, as we always work with well-defined maps, and all subsequent sets will be constructed accordingly.
\end{rem}

\begin{rem}
In the construction (\ref{askjfh;ashfd}) of $R^{[0]}$, for every $i < n$ we enforced that $a, b \in D_i$. From this point onward, we construct further levels $R^{[l]}$ based on the previous ones using the Mal'tsev operation, which preserves each $D_i$. Therefore, in what follows, instead of $a, b \in D_i$, we can simply write $a, b < q$.
\end{rem}

\begin{rem}
    For reacher languages $\Gamma$, where $s$ is the fixed maximum arity of relations in $\Gamma$ and $\#E^2_{\mathcal{X}} = m_2,\ldots,\#E^s_{\mathcal{X}} = m_s$, we would define a sequence of instances 
\begin{gather*}
    \mathcal{X}_0 = ([n], E^2_{\mathcal{X}_0},E^3_{\mathcal{X}_0},\ldots,E^s_{\mathcal{X}_0}),\ldots,\mathcal{X}_{m_2} = ([n], E^2_{\mathcal{X}_{m_2}},E^3_{\mathcal{X}_0},\ldots,E^s_{\mathcal{X}_0}),\\
    \mathcal{X}_{m_2+1} = ([n], E^2_{\mathcal{X}_{m_2}},E^3_{\mathcal{X}_1},\ldots,E^s_{\mathcal{X}_0}),\ldots,\mathcal{X}_{m_2+m_3} = ([n], E^2_{\mathcal{X}_{m_2}},E^3_{\mathcal{X}_{m_3}},\ldots,E^s_{\mathcal{X}_0}),\\
    \ldots \\
\mathcal{X}_{m_2+\ldots m_{s-1}+1} = ([n], E^2_{\mathcal{X}_{m_2}},E^3_{\mathcal{X}_{m_3}},\ldots,E^s_{\mathcal{X}_0}),\ldots,\mathcal{X}_{m_2+\ldots +m_s} = ([n], E^2_{\mathcal{X}_{m_2}},E^3_{\mathcal{X}_{m_3}},\ldots,E^s_{\mathcal{X}_{m_s}}).
    \end{gather*}    
We begin with the same set $R^{[0]}$, and construct it incrementally: first up to level $m_2$, then from level $m_2$ to $m_3$, and so on, from $m_{s-1}$ to $m_s$, using $s$ different functions $next^2, \ldots, next^s$. These functions are structurally similar, and each can be formalized within polynomially large sets, as will become clear later. Therefore, we will not return to this point in the subsequent discussion. 
\end{rem}

 \begin{lem}
      $V^1$ proves that $R^{[0]}$ generates $\mathscr{R}_{\Theta_0}$.
 \end{lem}

 \begin{proof}
     Consider any map $H$ from $[n]$ to $[q]$. We aim to generate $H$ from elements of the set $R^{[0]}$. Let $d_j = \min(D_j)$. Note that for every $i < n$ and every $a_i < q$ such that $H(i) = a_i$, we have $R^{[0ia_ia_i]} \neq \emptyset$. Moreover, for all $i, j < n$ and $b < q$, it holds that $R^{[0id_id_i]}(j,b)$ if and only if $b = d_j$.

Now define the set $Q$ as follows, for every $0 < t < n$:
\begin{equation}
    \begin{gathered}
    \forall j < n\, \forall b < q \quad Q^{[0]}(j,b) \iff R^{[00a_0a_0]}(j,b), \\
    \forall j < n\, \forall b < q \quad Q^{[t]}(j,b) \iff \mu(Q^{[t-1]}, R^{[0td_td_t]}, R^{[0ta_ta_t]}).
    \end{gathered}
\end{equation}

That is, $Q^{[0]}$ is a map from $R^{[0]}$ that sends $i = 0$ to $a_0$. To obtain $Q^{[1]}$, we apply $\mu$ to $Q^{[0]}$, to a map that sends all $i <n$ to $d_i$, and to a map that sends $i = 1$ to $a_1$. As a result, we obtain a map that sends all coordinates $1 < j < n$ to $d_j$, sends $i = 0$ to $a_0$, and $i = 1$ to $a_1$. We then repeat this process for all $i < n$ to construct a map in $Q$ that coincides with $H$. 

The existence of such a set follows from $\Sigma^{1,b}_1$-induction, and we have $H = Q^{[n-1]}$.
 \end{proof}

We now introduce Mal'tsev algorithm, already adapted to our setting (for a detailed description of the original version, we refer the reader to~\cite{doi:10.1137/050628957} and, for example, to \cite{https://doi.org/10.48550/arxiv.2210.07383}). The algorithm uses the function $solve(\Theta)$, the first line of which we have formalized above.

\begin{algorithm}[H]
\caption{: $solve(\Theta)$}\label{alg:cap}
\begin{algorithmic}[1] 
\For{ all $i<n, \forall a,b\in D_i,$ $\forall d_j=min(D_j)$} $R^{[0]}( i, a,b , i,a)\wedge \forall j\neq i<n \,R^{[0]}( i, a,b, j,d_j)$
\EndFor{}
\For{ all $1\leq l\leq m$} 
   \State $R^{[l]}:=next(R^{[l-1]}, min(E_{\mathcal{X}_l}), E^{[min(E_{\mathcal{X}_l})]}_{\ddot{\mathcal{A}}})$\Comment{\color{gray}add the minimum edge $(i,j)$ of  instance $E_{\mathcal{X}_l}$\bl}
\EndFor{}
\If{$R^{[m]}\neq \emptyset$}
   \Return{\text{yes}}
\Else{}
\Return{\text{no}}
\EndIf{}
\end{algorithmic}
\end{algorithm}

The only subroutine of $solve(\Theta)$ is the function $next(R^{[l-1]}, \min(E_{\mathcal{X}_l}), E^{[\min(E_{\mathcal{X}_l})]})$, which takes as input a compact representation $R^{[l-1]}$ of the solution set $\mathscr{R}_{\Theta_{l-1}}$, along with the minimal edge of the instance $\Theta_l$, and outputs a compact representation $R^{[l]}$ of the solution set $\mathscr{R}_{\Theta_l}$. To describe this function, we first introduce several of its subroutines. We provide a formalization for each of these immediately, to facilitate the reader's understanding.

\subsubsection{Formalization of the function $nonempty^1$}
The function $nonempty$ appears as a subroutine in the algorithm three times. Its general behavior is as follows: it takes as input a compact representation $R$ of some relation $\mathbb{R}$ that is invariant under a Mal'tsev term $\mu$, a sequence of indices $\{i_1, \ldots, i_j\} \subseteq \{0, \ldots, n-1\}$, and a $j$-ary relation $\S$, also invariant under $\mu$. The output is either a tuple $t \in \mathbb{R}$ such that $\pi_{i_1, \ldots, i_j}t \in \S$, or the value \emph{no}, indicating that no such tuple exists. In our case, $j$ is either $2$ or $3$.

We now describe and formalize the first occurrence of the function $nonempty$ (denoted $nonempty^1$). Here, the input consists of a compact representation $R^{[l-1]}$ of $\mathscr{R}_{\Theta_{l-1}}$ and a small relation $E^{[ij]}_{\ddot{\mathcal{A}}} \times \{a\}$, for some arbitrary $a \in A$. The output is either a tuple $T \in \mathscr{R}_{\Theta_{l-1}}$ such that the projection of $T$ onto $\{i,j,k\}$ belongs to $E^{[ij]}_{\ddot{\mathcal{A}}} \times \{a\}$, or the empty set if no such $T$ exists.

\begin{algorithm}[H]
\caption{: $nonempty^1(R^{[l-1]},  i,j,k, E^{[ij]}_{\ddot{\mathcal{A}}}\times\{a\})$}\label{alg:cap}
\begin{algorithmic}[1] 
\State $V:=R^{[l-1]}$
\While{$\exists T_1,T_2,T_3\in V$\text{ such that }$\pi_{i,j,k}(\mu(T_1,T_2,T_3))\notin \pi_{i,j,k}V$}
   \State $V:=V\cup \{\mu(T_1,T_2,T_2)\}$
\EndWhile{}\Comment{\color{gray}bounded by $q^3$\bl}
\If{$\exists T\in V$  such that $\pi_{i,j,k}(T)\in  E^{[ij]}_{\ddot{\mathcal{A}}}\times\{a\}$}
   \Return{T}
\Else{} 
    \Return{$\emptyset$}
\EndIf{}
\end{algorithmic}
\end{algorithm}
We need to show that the set $V$ can be generated using $\Sigma^{1,b}_1$-induction. To this end, for every $0 \leq t \leq q^3$, we define the set $V^{[t]}$ as follows. We fix the coordinates $i, j, k$ and the element $a$. For the initial layer $V^{[0]}$, we simply copy the set $R^{[l-1]}$, where $g$ is the index of a homomorphism in $R^{[l-1]}$ corresponding to some tuple $(i, a, b)$ that is witnessed by the homomorphism. Since each such $g$ corresponds to a tuple, it is bounded by $b_s$.
\begin{equation}
    \begin{gathered}
\forall g<b_s \forall r<n\forall c<q \,\, V^{[0]}( g,r,c)\iff R^{[l-1]}(g,r,c).       
    \end{gathered}
\end{equation}
For layers $1 \leq t \leq q^3$, we transfer all homomorphisms from the previous layer and add any new homomorphisms that can be generated using the function $\mu$ from those in layer $(t-1)$, provided they differ in at least one of the coordinates $i, j, k$ from those on layer $(t-1)$. To properly store a homomorphism generated from any three homomorphisms on the previous layer, we need to keep track of its identifier (as well as the identifiers of the homomorphisms used to produce it). If we apply $\mu$ to homomorphisms with identifiers $g_1, g_2, g_3$ and choose to store the resulting homomorphism, its identifier is set to $g = \langle g_1, g_2, g_3 \rangle$ (so that this identifier indeed refers to a valid homomorphism, not arbitrary values).

Therefore, we must compute an upper bound on such identifiers. To achieve this, in parallel with constructing the set $V$, we define a set $B$ that stores the upper bounds for each layer, based on the previous one. The function $seq(t, B)$ returns the minimal value stored in $B^{[t]}$.
\begin{equation}
    \begin{gathered}
        B(0,b_s)\wedge\forall b\neq b_s\,\neg B(0,b)\\
        \forall 0\leq t\leq q^3\,\,
        B(t+1,u) \iff u =\langle seq(t,B),seq(t,B),seq(t,B)\rangle.
    \end{gathered}
\end{equation}
Since $\langle b_s, b_s, b_s \rangle < (3b_s + 1)^8$, it follows that $seq(q^3 + 1, B) \leq c \cdot b_s^{8^{(q^3 + 1)}}$ for some appropriate constant $c$. Therefore, the final upper bound remains polynomial in the number of variables.
\begin{equation}\label{aklsdjfglg}
    \begin{gathered}
\forall g<seq(t+1,B) \forall r<n\forall b<q \,\, V^{[t+1]} ( g, r,b) \iff V^{[t]} (g,r,b) \vee\\ \vee\big[\exists g_1,g_2,g_3<seq(t,B) \forall g'<seq(t,B)\wedge g=\langle g_1,g_2,g_3\rangle\wedge \\
\mu(V^{[tg_1]},V^{[tg_2]},V^{[tg_3]})( r,b)\wedge \\
    \big(\mu(V^{[tg_1]},V^{[tg_2]},V^{[tg_3]})(i)\neq V^{[tg']}(i)\vee\mu(V^{[tg_1]},V^{[tg_2]},V^{[t,g_3]})(j)\neq V^{[tg']}(j)\vee\\
\vee\mu(V^{[tg_1]},V^{[tg_2]},V^{[tg_3]})(k)\neq V^{[tg']}(k)\big)\big].
    \end{gathered}
\end{equation}
Finally, in the last layer $t = (q^3 + 1)$, we store only those homomorphisms from $V^{[q^3]}$ whose projection onto the coordinates $i, j, k$ belongs to $E^{[ij]}_{\ddot{\mathcal{A}}} \times \{a\}$.
\begin{equation}
    \begin{gathered}
\forall g<seq(q^3+1,B) \forall r<n\forall b<q \,\, V^{[q^3+1]} ( g\bl, r,b) \iff V^{[q^3]} (g,r,b)\wedge \\
\wedge (\forall a_i,a_j<q\,\,V^{[q^3g]}(i)=a_i\wedge V^{[q^3g]}(j) =a_j\rightarrow E^{[ij]}_{\ddot{\mathcal{A}}}(a_i,a_j))\wedge V^{[q^3g]}(k)=a.
    \end{gathered}
\end{equation}
The existence of both sets $V$ and $B$ follows from $\Sigma^{1,b}_1$-induction. We now define the function $nonempty^1$ via its bit-definition, for all $r < n$ and $b < q$:
   \begin{equation*}
       \begin{gathered}
           nonempty^1(R^{[l-1]}, i,j,k, E^{[ij]}_{\ddot{\mathcal{A}}}\times\{a\})(r,b) = T(r,b)\iff \exists g<seq(q^3+1,B)\\
       V^{[q^3+1][g]}(r)=b \wedge \forall g'<seq(q^3+1,B)\,\, (g'<g \rightarrow V^{[q^3+1][g']} = \emptyset).
       \end{gathered}
   \end{equation*}
That is, we always return a homomorphism from $V^{[q^3+1]}$ with the minimal identifier $g$ that satisfies the required projection onto the coordinates $i, j, k$. Hence, $nonempty^1$ is a well-defined function that returns a unique value determined by its arguments. Note also that if no such $g$ exists, then $nonempty^1$ returns the empty set.
\begin{lem}[Correctness of $nonempty^1$]\label{djksfasgfljfsf}
Assume that $R^{[l-1]} \subseteq \mathscr{R}_{\Theta_{l-1}}$ and $\mathrm{Sig}(R^{[l-1]}) = \mathrm{Sig}(\mathscr{R}_{\Theta_{l-1}})$. Then 
$V^1$ proves that $nonempty^1(R^{[l-1]}, i, j, k, E^{[ij]}_{\ddot{\mathcal{A}}} \times \{a\})$ outputs a homomorphism $T$ from $\mathcal{X}_{l-1}$ to $\ddot{\mathcal{A}}$ such that $\pi_{i,j,k}T \subseteq E^{[ij]}_{\ddot{\mathcal{A}}} \times \{a\}$, or, if 
\[
nonempty^1(R^{[l-1]}, i, j, k, E^{[ij]}_{\ddot{\mathcal{A}}} \times \{a\}) = \emptyset,
\]
then there is no homomorphism $T \in \mathscr{R}_{\Theta_{l-1}}$ such that $\pi_{i,j,k}T \subseteq E^{[ij]}_{\ddot{\mathcal{A}}} \times \{a\}$.
\end{lem}
\begin{proof}
    Since $V^{[0]} = R^{[l-1]}$ and $R^{[l-1]} \subseteq \mathscr{R}_{\Theta_{l-1}}$, it follows that for every $g < b_s$, the map $V^{[0g]}$ is a homomorphism from $\mathcal{X}$ to $\ddot{\mathcal{A}}$. Moreover, for every $0 < t \leq q^3$, we add to $V^{[t]}$ only those maps that are generated from homomorphisms using the function $\mu$. By Lemma~\ref{simplemm}, it follows that for every $g < \mathrm{seq}(q^3, B)$, the map $V^{[q^3g]}$ is also a homomorphism, and we have $V^{[(q^3+1)g]} \subseteq V^{[q^3g]}$.

    Now assume that there exists a homomorphism $T \in \mathscr{R}_{\Theta_{l-1}}$ such that $T(i) = b$, $T(j) = c$, and $E^{[ij]}_{\ddot{\mathcal{A}}}(b, c)$, and further $T(k) = a$. Since $\mathrm{Sig}(R^{[l-1]}) = \mathrm{Sig}(\mathscr{R}_{\Theta_{l-1}})$, it follows that there exists a set $S$ such that $\mathrm{Sig}_{CR}(S, R) \wedge \mathrm{Sig}_{SS}(S, \Theta)$. This implies that $S(i, b, b)$ holds, and therefore $R^{[(l-1)ibb]}(i) = b$. Consequently, $V^{[0ibb]}$ contains a homomorphism with $V^{[0ibb]}(i) = b$.

We can now formalize the same construction as in Theorem~\ref{alalksdjasd}. Since we are only concerned with the projection onto three coordinates, the required homomorphism can be constructed and stored in $V$ in just two steps.

\end{proof}

\subsubsection{Formalization of the function $fixvalues$}
The next function we consider is called $fixvalues$. This function takes as input a compact representation $R$ of a relation $\mathbb{R}$ that is invariant under a Mal'tsev term $\mu$, along with a sequence of values $a_0, \ldots, a_{k-1}$ from the domain $A$, for some $k \leq n - 1$. It returns a compact representation $U$ of the relation defined by
$$
\mathbb{U} = \{t \in \mathbb{R} : \pi_0 t = a_0, \ldots, \pi_{k-1} t = a_{k-1}\}.
$$
In the main algorithm, we use this function for any tuple $(k, a, b) \in [n] \times D_k^2$ after the execution of the function $nonempty^1$, in the case where it outputs some element $T_a$, in order to find a corresponding element $T_b$ such that together they witness the tuple.

For convenience, we start the indexing of the sequence of sets $U_j$ from $U_{-1}$ (since we add the constraint $a_j$ at level $U_j$). For every $j < k$, each $U_j$ is a compact representation of the relation
$$
\mathbb{U} = \{t \in \mathbb{R} : \pi_0 t = a_0, \ldots, \pi_j t = a_j\}.
$$
\begin{algorithm}[H]\label{asjd643}
\caption{$fixvalues(R^{[l-1]}, a_0, \ldots, a_{k-1})$}\label{aaasp}
\begin{algorithmic}[1] 
\State $U_{-1} := R^{[l-1]}$ \Comment{\color{gray}set $U_{-1} = R^{[l-1]}$; for every $0 \leq j \leq k-1$, gradually add the constraint $x_j = a_j$ to the previous set $U_{j-1}$ to obtain a new set $U_j$, which is a compact representation of the relation $\{T \in \mathscr{R}_{\Theta_{l-1}} : T(0) = a_0, \ldots, T(j) = a_j\}$\bl}

\For{$j$ from $0$ to $k-1$}
   \If{$(j, a_j, a_j) \notin \mathrm{Sig}(U_{j-1})$}
   \State $U_j := \emptyset$ \Comment{ \color{gray}if $\mathrm{Sig}(U_{j-1})$ does not contain such a tuple, then there is no solution $T$ to the instance with $T(0) = a_0, \ldots, T(j) = a_j$, and we eventually return an empty set $U_{k-1}$\bl}
   \Else
   \State $U_j := \{T\}$ where $T \in U_{j-1}$ and $T(j) = a_j$ \Comment{\color{gray}in each $U_{j-1}$, if $U_{j-1} \neq \emptyset$, all elements satisfy $T(0) = a_0, \ldots, T(j-1) = a_{j-1}$\bl}
     \For{all $(i, a, b) \in \mathrm{Sig}(U_{j-1})$ with $i > j$}
        \State Let $T_a, T_b \in U_{j-1}$ witness the tuple $(i, a, b)$; if $a = b$, choose $T_a = T_b$.
        \State Let $T := nonempty^3(U_{j-1}, j, i, \{(a_j, a)\})$ \Comment{\color{gray} find an element $T$ in the relation $\{T \in \mathscr{R}_{\Theta_{l-1}} : T(0) = a_0, \ldots, T(j-1) = a_{j-1}\}$ such that $T(j) = a_j$ and $T(i) = a$\bl}
        \If{$T \neq \emptyset$} 
           \State $U_j := U_j \cup \{T, \mu(T, T_a, T_b)\}$ \Comment{\color{gray} if there exists an element $T$ in the relation $\{T \in \mathscr{R}_{\Theta_{l-1}} : T(0) = a_0, \ldots, T(j-1) = a_{j-1}\}$ such that $T(j) = a_j$ and $T(i) = a$, then add $T$ to $U_j$, along with the value obtained by applying $\mu$ to $T, T_a, T_b$; this ensures that the pair $(T, \mu(T, T_a, T_b))$ witnesses the tuple $(i, a, b)$\bl}
        \EndIf
     \EndFor
   \EndIf
\EndFor
\\
\Return{$U_{k-1}$}
\end{algorithmic}
\end{algorithm}
Note that in each iteration of lines $7-13$, we add at most two elements to $U_j$. Therefore, $U_j$ remains a compact representation. 

We now need to show that the set $U_{k-1}$ can be formalized for any $R^{[l-1]}$, any $k \leq n - 1$, and any sequence $a_0, \ldots, a_{k-1}$ using $\Sigma^{1,b}_1$-induction. Note that since this set is later used in the procedure $nonempty^2$ (which is structurally similar to $nonempty^1$), it must have the same form as $R^{[l-1]}$ itself - that is, for every pair of tuples in the signature $(i, a, b)$ and $(i, b, a)$, it must contain corresponding unique maps $T_a$ and $T_b$ under indices $\langle i, a, b\rangle$ and $\langle i, b, a\rangle$ that witness them.

Thus, the function $fixvalues$ must return the set $U^{[k-1]}$. We cannot define $fixvalues$ exactly as described in the algorithm because the number of arguments to the function is variable and can be as large as $n$. However, since we use it in the function $next$ for any tuple $(k, a, b)$ to find a pair for some map $H_a$ that witnesses the tuple, we can simply pass the following set of arguments to $fixvalues$:
\begin{equation}
    \begin{gathered}
        fixvalues(R^{[l-1]}, H_a, k-1) = U^{[k-1]}.
    \end{gathered}
\end{equation}
In this way, we can access the values $a_0, \ldots, a_{k-1}$ through a single set as $H_a(0), \ldots, H_a(k-1)$. In the subsequent formalization, for simplicity of presentation, we will use the first option.

Let us first describe the function $nonempty^3(U^{[j-1]}, j, i, \{(a_j, a)\})$, which we use as a subroutine at every level $j < k$ for the set $U$. The function returns either a map $T$ from $\mathbb{U}^{[j-1]}$ such that $\pi_{j,i}T = (a_j, a)$, or the empty set. In fact, its definition differs only slightly from that of the function $nonempty^1$. Instead of the ternary relation $E^{[ij]}_{\ddot{\mathcal{A}}} \times \{a\}$ on coordinates $i, j, k$, here we work with a binary relation $(a_j, a)$ on coordinates $j, i$. The construction of the set $V$ will be bounded by level $(q^2 + 1)$.

\begin{equation*}
\begin{gathered}
nonempty^3(U^{[j-1]}, j, i, \{(a_j, a)\})(r, b) = T(r, b) \iff \exists g < seq(q^2 + 1, B) \\
V^{[(q^2 + 1)g]}(r) = b \wedge \forall g' < seq(q^2 + 1, B)\,\, (g' < g \rightarrow V^{[(q^2 + 1)g']} = \emptyset).
\end{gathered}
\end{equation*}
Thus, to construct the set $U^{[k-1]}$, for every level $j < k$, for every tuple $(i,a,b)\in Sig(U_{j-1})$ we simultaneously construct a set $V^{[(q^2 + 1)]}$ and a set of upper bounds $B^{[(q^2 + 1)]}$. 
For every layer $j < k$, we run $nonempty^3$ on at most $n \times q^2$ tuples. Thus, for a single run of $fixvalues$, we can store both $V^{[(q^2 + 1)]}$ and $B^{[(q^2 + 1)]}$ in two supersets, namely $V^{[k-1][nqq][q^2 + 1]}$ and $B^{[k-1][nqq][q^2 + 1]}$. Both sets are bounded by a polynomial in $n$, and since their definitions use only $\Sigma^{1,b}_0$-formulas, their existence follows from $\Sigma^{1,b}_1$-induction. Therefore, we will not explicitly repeat the construction for $nonempty^3$.

Let us now move to the definition of the set $U$. For the initial layer $U^{[-1]}$, we copy the set $R^{[l-1]}$. (While we cannot formally use $-1$ as an index in the formalization, this can be handled by a simple shift by $1$.)
\begin{equation}
    \begin{gathered}
\forall i<n\forall a, b<q \bl\forall r<n\forall c<q \,\, U^{[-1]}( i,a,b\bl,r,c)\iff R^{[l-1]}( i,a,b\bl,r,c).       \end{gathered}
\end{equation} 
For each $U^{[j]}$, under the index $\langle j, a_j, a_j\rangle $ we copy the map $T$ from $U^{[j-1]}$ using the same identifier. The reason is that on the previous level $(j-1)$, when we added the constraint $x_{j-1} = a_{j-1}$, if there existed a map $T$ in $\mathbb{U}^{[j-2]}$ with $T(j-1) = a_{j-1}$ and $T(j) = a_j$, then the function $nonempty^3(U^{[j-2]}, j-1, j, \{(a_{j-1}, a_j)\})$ already returned it, and we stored that map under the index $\langle j, a_j, a_j\rangle$.
\begin{equation}
    \begin{gathered}
     \bl\forall r<n\forall c<q \,\, U^{[j]}(j,a_{j},a_{j}, r, c) \iff U^{[j-1]}(j,a_{j},a_{j}, r, c).
    \end{gathered}
\end{equation}
Note also that by construction, any map $T$ in $U^{[j-1]}$ satisfies $T(0) = a_0, \ldots, T(j-1) = a_{j-1}$. Since we require $U^{[j]}$ to be a compact representation, for every $i < j$ and tuple $(i, a_i, a_i)$, it suffices to store a single map in $U^{[j]}$. We can choose this map to be the same one stored under the index $\langle j, a_j, a_j\rangle$.
\begin{equation}
    \begin{gathered}
        \forall i<j\forall r<n\forall c<q\,\,U^{[j]}(i, a_i, a_i, r,c ) \iff U^{[j-1]}(j, a_j,a_j, r,c ).
    \end{gathered}
\end{equation}
Furthermore, for every $i \leq j$ and every tuple $(i, a, b)$ such that $a, b \neq a_i$, the set $U^{[j]}$ cannot contain any additional maps, as such tuples cannot appear in the signature.
\begin{equation}
    \begin{gathered}
        \forall i\leq j \forall a,b< q \forall r<n\forall c<q\,\,\,\,a,b\neq a_i \rightarrow \neg U^{[j]}(i, a, b, r,c ).
    \end{gathered}
\end{equation}
Finally, we need to store in $U^{[j]}$ all maps that witness the indices $(i, a, b)$ and $(i, b, a)$ for $i > j$, such that the maps send $j$ to $a_j$. To achieve this, we consider the maps $T_a$ and $T_b$ from $U^{[j-1]}$, stored under the same indices, and check whether we can generate new maps $T_a'$ and $T_b'$ from them.

First, we apply the function $nonempty^3(U^{[j-1]}, j, i, \{(a_j, a)\})$ to obtain a map $T$ such that $T(j) = a_j$ and $T(i) = a$. Note that since $U^{[j-1]}$ is a representation of the set $\mathbb{U}^{[j-1]}$, and since $\mathbb{U}^{[j]} \subseteq \mathbb{U}^{[j-1]}$, if such a map exists, the function $nonempty^3$ will find it.

We then apply $\mu(T, T_a, T_b)$, which yields a map satisfying $\mu(T, T_a, T_b)(j) = a_j$ and $\mu(T, T_a, T_b)(i) = b$. Since we intend to store both $T$ and $\mu(T, T_a, T_b)$ in $U^{[j]}$, and since they are generated via different constructions, we require a case distinction. We begin with the case $a \leq b$:
\begin{equation}
    \begin{gathered}
    \forall j<i<n\forall a\leq b< q \forall r<n\forall c<q \,\, U^{[j]}(i,a,b, r, c) \iff \forall r_1,r_2<n\exists c_1,c_2 <q\\
    U^{[(j-1)iab]}(r_1) = c_1\wedge U^{[(j-1)iba]}(r_2) = c_2 \wedge \\
 nonempty^3(U^{[j-1]}, j,i,\{(a_{j},a)\})(r,c).\\
    \end{gathered}
\end{equation}
The second case for $b>a$:
\begin{equation}
    \begin{gathered}
    \forall j<i<n\forall a< b< q \forall r<n\forall c<q \,\, U^{[j]}(i,b,a, r, c) \iff \forall r_1,r_2<n\exists c_1,c_2 <q\\
    U^{[(j-1)iab]}(r_1) = c_1\wedge U^{[(j-1)iba]}(r_2) = c_2 \wedge \\
\mu(nonempty^3(U^{[j-1]}, j,i,\{(a_{j},a)\}), U^{[(j-1)iab]},  U^{[(j-1)iba]})(r,c).
    \end{gathered}
\end{equation}
The second lines in both definitions indicate that there are indeed two maps, $U^{[(j-1)iab]}$ and $U^{[(j-1)iba]}$, from the previous layer, stored under the unique indices $\langle i, a, b\rangle$ and $\langle i, b, a\rangle$, which witness these two tuples. Thus, we store $T$ under the index $\langle i, a, b \rangle$ for $a \leq b$, and $\mu(T, T_a, T_b)$ under the index $\langle i, b, a \rangle$ for $b > a$.

This approach works due to the uniqueness of indices, the fact that $nonempty^3$ is a well-defined function, and the validity of Remarks~\ref{a=++++},~\ref{&&^*&}. Note that if $nonempty^3(U^{[j-1]}, j, i, \{(a_j, a)\}) = T$ returns the empty set, then $\mu$ will also return the empty set, as $T$ will not be a well-defined map. In that case, we do not store any map at level $U^{[j]}$ under the indices $\langle i, a, b\rangle$ and $\langle i, b, a\rangle$.

This concludes the construction of the set $U^{[k-1]}$. In the definition, we used only $\Sigma^{1,b}_0$-formulas; therefore, the existence of such a set follows from $\Sigma^{1,b}_1$-induction. Moreover, the set is unique, since $nonempty^3$ is a well-defined function. Consequently, $fixvalues$, which returns $U^{[k-1]}$, is also a well-defined set function.

Now let us prove the correctness of the function $fixvalues$. We need to show that if $fixvalues$ takes as input a compact representation $R^{[l-1]}$ of $\mathscr{R}_{\Theta_{l-1}}$, a map $H_a$, and a number $(k-1)$, then it outputs a compact representation of $\mathscr{R}_{\Theta'_{l-1}}$, where $\Theta'_{l-1}$ is the same as the instance $\Theta_{l-1}$ except that for all $i < k$, the domains are restricted to $D_i' = \{H_a(i)\}$. We denote $U^{[-1]} = R^{[l-1]}$ as the compact representation of the instance $\mathscr{U}_{\Theta'_{-1}} = \mathscr{R}_{\Theta_{l-1}}$. For each $0 \leq j \leq k - 1$, we define a CSP instance $\Theta'_j$ such that $\mathcal{X}'_j = \mathcal{X}_{l-1}$ and $\ddot{\mathcal{A}}'_j$ satisfies the following domain conditions:
\[
D_0' = \{a_0\}, \ldots, D_j' = \{a_j\}, \quad D_{j+1}' = D_{j+1}, \ldots, D_{n-1}' = D_{n-1},
\]
where $H_a(0) = a_0, \ldots, H_a(j) = a_j$. Additionally, for all $s, t < n$, we require that
\begin{equation}
    \begin{gathered}
        E_{\ddot{\mathcal{A}}'_j}^{[st]}(a, b) \rightarrow E_{\ddot{\mathcal{A}}}^{[st]}(a, b) \wedge D'_s(a) \wedge D'_t(b).
    \end{gathered}
\end{equation}
The following lemma is proved in a way similar to Lemma~\ref{djksfasgfljfsf}.

\begin{lem}[Correctness of $nonempty^3$]\label{=a=-sify34gtth}
Assume that $U^{[j-1]} \subseteq \mathscr{U}_{\Theta'_{j-1}}$ and $\mathrm{Sig}(U^{[j-1]}) = \mathrm{Sig}(\mathscr{U}_{\Theta'_{j-1}})$. Then 
$V^1$ proves that the function $nonempty^3(U^{[j-1]}, j, i, \{(a_j, a)\})$ for any $i > j$, any $a_j \in D_j$, and any $a \in D_i$ either outputs a homomorphism $T$ from $\mathcal{X}'_{j-1}$ to $\ddot{\mathcal{A}}'$ such that $\pi_{j,i}T = (a_j, a)$, or, if
\[
nonempty^3(U^{[j-1]}, j, i, \{(a_j, a)\}) = \emptyset,
\]
then there is no homomorphism $T \in \mathscr{U}_{\Theta'_{j-1}}$ such that $\pi_{j,i}T = (a_j, a)$.
\end{lem}

\begin{lem}[Correctness of $fixvalues$]\label{alksjdca;syfrg}
Assume that $R^{[l-1]} \subseteq \mathscr{R}_{\Theta_{l-1}}$ and $\mathrm{Sig}(R^{[l-1]}) = \mathrm{Sig}(\mathscr{R}_{\Theta_{l-1}})$, $H_a\in \mathscr{R}_{\Theta_{l-1}}$, and $k < n$. Then $V^1$ proves that the function $fixvalues(R^{[l-1]}, H_a, k-1)$ outputs a compact representation $U^{[k-1]}$ of the instance $\mathscr{U}_{\Theta'_{k-1}}$, i.e., $U^{[k-1]} \subseteq \mathscr{U}_{\Theta'_{k-1}}$ and $\mathrm{Sig}(U^{[k-1]}) = \mathrm{Sig}(\mathscr{U}_{\Theta'_{k-1}})$.
\end{lem}

\begin{proof}
First, by construction, we have $U^{[-1]} \subseteq \mathscr{U}_{\Theta'_{-1}}$ and $\mathrm{Sig}(U^{[-1]}) = \mathrm{Sig}(\mathscr{U}_{\Theta'_{-1}})$. To prove the lemma, it suffices to show that for every $0 \leq j < k-1$, if 
\[
U^{[j-1]} \subseteq \mathscr{U}_{\Theta'_{j-1}} \quad \text{and} \quad \mathrm{Sig}(U^{[j-1]}) = \mathrm{Sig}(\mathscr{U}_{\Theta'_{j-1}}),
\]
then
\[
U^{[j]} \subseteq \mathscr{U}_{\Theta'_{j}} \quad \text{and} \quad \mathrm{Sig}(U^{[j]}) = \mathrm{Sig}(\mathscr{U}_{\Theta'_{j}}).
\]

By the construction of the instances, it is clear that $\mathscr{U}_{\Theta'_{j}} \subseteq \mathscr{U}_{\Theta'_{j-1}}$, and hence $\mathrm{Sig}(\mathscr{U}_{\Theta'_{j}}) \subseteq \mathrm{Sig}(\mathscr{U}_{\Theta'_{j-1}})$. The only maps that can belong to $\mathscr{U}_{\Theta'_{j}}$ are those that send $0$ to $a_0$, \ldots , and $j$ to $a_j$. Therefore, if $\mathscr{U}_{\Theta'_{j}}$ is non-empty, the tuple $(j, a_j, a_j)$ must belong to $\mathrm{Sig}(U^{[j-1]})$. We copy to $U^{[j]}$ a map $T$ with this index only if such a map exists in $U^{[j-1]}$. Since $U^{[j-1]} \subseteq \mathscr{U}_{\Theta'_{j-1}}$ and $T(j) = a_j$, it follows that $T \in \mathscr{U}_{\Theta'_{j}}$. This is the only map we copy to $U^{[j]}$ from $U^{[j-1]}$ for $i \leq j$. Note that no other maps are needed for these coordinates, as we do not need to generate additional values.

For all $(i, a, b) \in \mathrm{Sig}(\mathscr{U}_{\Theta'_{j-1}})$ with $i > j$, we store in $U^{[j i a b]}$ the map $$nonempty^3(U^{[j-1]}, j, i, \{(a_j, a)\}),$$ and in $U^{[j i b a]}$ the map 
\[
\mu(nonempty^3(U^{[j-1]}, j, i, \{(a_j, a)\}), U^{[(j-1) i a b]}, U^{[(j-1) i b a]}).
\]
In the first case, by Lemma~\ref{=a=-sify34gtth}, this is a homomorphism $T$ from $\mathscr{U}_{\Theta'_{j-1}}$ such that $\pi_{j i} T = (a_j, a)$, and hence $T \in \mathscr{U}_{\Theta'_{j}}$. In the second case, by Lemma~\ref{simplemm}, the result is a homomorphism $T$ from $\mathscr{U}_{\Theta'_{j}}$ such that $\pi_{j i} T = (a_j, b)$. Thus, $U^{[j]} \subseteq \mathscr{U}_{\Theta'_{j}}$. The equality $\mathrm{Sig}(U^{[j]}) = \mathrm{Sig}(\mathscr{U}_{\Theta'_{j}})$ follows from the correctness of the function $nonempty^3$.
\end{proof}

 \subsubsection{Formalization of the function $next$}\label{+++)_)_)KJ}

It remains to formalize the main procedure of the algorithm, which incorporates all the previously defined subroutines. This function takes a compact representation $R^{[l-1]}$ of the solution set $\mathscr{R}_{\Theta_{l-1}}$, extracts the minimal edge of the instance $\Theta_l$, and outputs a compact representation $R^{[l]}$ of the solution set $\mathscr{R}_{\Theta_l}$.

\begin{algorithm}[H]
\caption{: $next(R^{[l-1]}, \langle i,j\rangle, E^{[\langle i,j\rangle]}_{\ddot{\mathcal{A}}})$}\label{alg:cap}
\begin{algorithmic}[1] 
\State $W:=\emptyset$
\For{ all $(k,a,b)\in n\times D_i^2$} 
  \State $H_a:=nonempty^1(R^{[l-1]},  i,j\,k, E^{[ij]}_{\ddot{\mathcal{A}}}\times\{a\})$
  \If{$H_a\neq \emptyset $}
       \State $H_b:=nonempty^2(fixvalues(R^{[l-1]},H_a, k-1),  i,j,k, E^{[ij]}_{\ddot{\mathcal{A}}}\times\{b\})$
       \If{$H_b\neq \emptyset$}
          \State $W:=W\cup\{H_a,H_b\} $      
       \EndIf{}
  \EndIf{}  
\EndFor \\

\Return{$W$}
\end{algorithmic}
\end{algorithm}

We note that in the original paper~\cite{doi:10.1137/050628957}, this function is referred to as $next-beta$ and is used as a subroutine within a broader function called $next$. The separation was introduced to avoid potentially exponential running time, as the original setting did not restrict the arity of constraints. However, this concern does not arise in our case, so we simplify the construction and refer to this function directly as $next$.

The formalization of this function is straightforward and builds directly on the components introduced above. The function $nonempty^2$ is defined in exactly the same way as $nonempty^1$, with the sole difference that in $V^{[0]}$ we copy the set $U^{[k-1]}$ returned by the function $fixvalues$. Then the function $next(R^{[l-1]}, \langle i,j\rangle, E^{[\langle i,j\rangle]})$ returns the set $W$ where
\begin{equation}
    \begin{gathered}
\forall i<n\forall a \leq b< q \forall r<n\forall c<q \\W(i,a,b, r, c) \iff nonempty^1(R^{[l-1]},  i,j\,k, E^{[ij]}_{\ddot{\mathcal{A}}}\times\{a\})(r,c)\wedge \\
nonempty^2(fixvalues(R^{[l-1]},nonempty^1(R^{[l-1]},  i,j\,k, E^{[ij]}_{\ddot{\mathcal{A}}}\times\{a\}), k-1), \\ i,j,k, E^{[ij]}_{\ddot{\mathcal{A}}}\times\{b\}) \neq \emptyset.
    \end{gathered}
\end{equation}
and 
\begin{equation}
    \begin{gathered}
\forall i<n\forall a < b< q \forall r<n\forall c<q \\ W(i,b,a, r, c) \iff nonempty^1(R^{[l-1]},  i,j\,k, E^{[ij]}_{\ddot{\mathcal{A}}}\times\{a\})\neq \emptyset\\
nonempty^2(fixvalues(R^{[l-1]},nonempty^1(R^{[l-1]},  i,j\,k, E^{[ij]}_{\ddot{\mathcal{A}}}\times\{a\}), k-1), \\ i,j,k, E^{[ij]}_{\ddot{\mathcal{A}}}\times\{b\})(r,c).
    \end{gathered}
\end{equation}
We again use case distinction since $H_a$ and $H_b$ are constructed differently. The function $next$ is well-defined, as all of its constituent functions are well-defined. The formalization of this function completes the formalization of the algorithm.

To carry out the formalization, we need to construct only a few sets:
\begin{itemize}
    \item $V_1$, $B_1$ for $nonempty^1$; $V_2$, $B_2$ for $nonempty^2$; and $V_3$, $B_3$ for $nonempty^3$,
    \item $U$ for $fixvalues$, and
    \item $W$ for $next$.
\end{itemize}

All of these sets are constructed in parallel using $\Sigma^{1,b}_1$-induction, and each new layer of every set depends only on the previous layers of that set or on earlier layers of other sets. For example, the set $W^{[m]}$ consists of only $m$ layers. The innermost set, $V_3$, has the following structure: 
\[
V_3^{[m][nqq][n][nqq][q^2+1]}.
\]
This reflects that, during a single run of $nonempty^3$, we construct at most $(q^2 + 1)$ layers of the set $V$. For each inner call to $fixvalues$, we invoke $nonempty^3$ for at most $n \times q^2$ tuples, across at most $n$ values. Subsequently, we call $nonempty^2$ for up to $n \times q^2$ tuples, over at most $m$ layers of $W$.

It is clear that the sizes of all these sets are polynomial in the number of variables. Therefore, this construction remains valid for instances of arbitrary size.

\begin{lem}[Correctness of $nonempty^2$]\label{lalaskksdjjd}
Assume that $U^{[k-1]} \subseteq \mathscr{U}_{\Theta'_{k-1}}$ and $\mathrm{Sig}(U^{[k-1]}) = \mathrm{Sig}(\mathscr{U}_{\Theta'_{k-1}})$. Then 
$V^1$ proves that $nonempty^2(U^{[k-1]}, i, j, k, E^{[ij]}_{\ddot{\mathcal{A}}} \times \{a\})$ either outputs a homomorphism $T$ from $\mathcal{X}'_{k-1}$ to $\ddot{\mathcal{A}}'_{k-1}$ such that $\pi_{i, j, k}T \subseteq E^{[ij]}_{\ddot{\mathcal{A}}} \times \{a\}$, or, if
\[
nonempty^2(U^{[k-1]}, i, j, k, E^{[ij]}_{\ddot{\mathcal{A}}} \times \{a\}) = \emptyset,
\]
then there is no homomorphism $T \in \mathscr{U}_{\Theta'_{k-1}}$ such that $\pi_{i, j, k}T \subseteq E^{[ij]}_{\ddot{\mathcal{A}}} \times \{a\}$.
\end{lem}

The following theorem follows immediately from Lemmas~\ref{djksfasgfljfsf},~\ref{=a=-sify34gtth},~\ref{alksjdca;syfrg}, and~\ref{lalaskksdjjd}.

\begin{thm}[Correctness of $next$]\label{alalksldkjk}
Assume that $R^{[l-1]} \subseteq \mathscr{R}_{\Theta_{l-1}}$ and $\mathrm{Sig}(R^{[l-1]}) = \mathrm{Sig}(\mathscr{R}_{\Theta_{l-1}})$. Then 
$V^1$ proves that 
\[
next(R^{[l-1]}, \langle i, j \rangle, E^{[\langle i, j \rangle]}_{\ddot{\mathcal{A}}})
\]
outputs a compact representation of $\mathscr{R}_{\Theta_l}$, that is, a set $R^{[l]} \subseteq \mathscr{R}_{\Theta_l}$ such that $\mathrm{Sig}(R^{[l]}) = \mathrm{Sig}(\mathscr{R}_{\Theta_l})$.

\end{thm}

\subsection{Mal’tsev Upper Bound}\label{alskdjf';lahdg'a}

\begin{thm}\label{jkjkajsflageee}
For any relational structure $\mathcal{A}$ that corresponds to an algebra with a Mal'tsev operation, theory $V^1$ proves the soundness of Mal'tsev algorithm:
\[
V^1 \vdash \forall \mathcal{X} \, \forall \ddot{\mathcal{A}} \, \forall R  \;\; \left( \mathrm{Comp}(R, \mathcal{X}, \ddot{\mathcal{A}}) \wedge R^{[\#E_{\mathcal{X}}]} = \emptyset \longrightarrow \neg \mathrm{Hom}(\mathcal{X}, \ddot{\mathcal{A}}) \right),
\]
where
\begin{gather*}
\mathrm{Comp}(R, \mathcal{X}, \ddot{\mathcal{A}}) \iff \forall i < n \; \forall a, b \in V^{[i]}_{\ddot{\mathcal{A}}} \;\;
R^{[0]}(i, a, b\bl, i, a\bl) \wedge \forall j \neq i < n \;\; R^{[0]}(i, a, b\bl, j, d_j\bl), \\
\forall 0 < l \leq \#E_{\mathcal{X}} \;\; R^{[l]} = next(R^{[l-1]}, \min(E_{\mathcal{X}_l}), E^{[\min(E_{\mathcal{X}_l})]}_{\ddot{\mathcal{A}}}).
\end{gather*}
\end{thm}
\begin{proof}
Consider any unsatisfiable CSP instance $\Theta = (\mathcal{X}, \ddot{\mathcal{A}})$ corresponding to an algebra with a Mal'tsev operation. Let $R$ be the set generated as in Section \ref{alskhf;shgf} during the execution of the algorithm on input $\mathcal{X}$. Then the fact that $R^{[0]} \subseteq \mathscr{R}_{\Theta_0}$ and $\mathrm{Sig}(R^{[0]}) = \mathrm{Sig}(\mathscr{R}_{\Theta_0})$ follows directly from the construction. Lemmas~\ref{djksfasgfljfsf},~\ref{=a=-sify34gtth},~\ref{lalaskksdjjd},~\ref{alksjdca;syfrg}, and~\ref{alalksldkjk} establish the correctness of the procedures $nonempty^1$, $nonempty^2$, $nonempty^3$, $fixvalues$, and $next$. Thus, theory $V^1$ proves the following: if $R^{[l-1]} \subseteq \mathscr{R}_{\Theta_{l-1}}$ and $\mathrm{Sig}(R^{[l-1]}) = \mathrm{Sig}(\mathscr{R}_{\Theta_{l-1}})$, then
$$next(R^{[l-1]}, \min(E_{\mathcal{X}_l}), E^{[\min(E_{\mathcal{X}_l})]}_{\ddot{\mathcal{A}}})$$
outputs a set $R^{[l]} \subseteq \mathscr{R}_{\Theta_l}$ such that $\mathrm{Sig}(R^{[l]}) = \mathrm{Sig}(\mathscr{R}_{\Theta_l})$.

Now, if $R^{[\#E_{\mathcal{X}}]} = \emptyset$, then from formula~\eqref{eq1232346}, it follows that for any set $S$ satisfying $\mathrm{Sig}_{CR}(S, R^{[\#E_{\mathcal{X}}]})$, we have $\neg S(i, a, b)$ for all $i < n$ and all $a, b \in D_i$. Therefore, from formulas~\eqref{eqj3jkr} and~\eqref{akajsgdfl7}, we obtain:
\begin{gather*}
\forall i < n \, \forall a<q \, \forall b < q \;\; 
\forall T_a, T_b \leq b_h \\
T_a(i) \neq a \vee T_b(i) \neq b \vee \exists j < i \;\; T_a(j) \neq T_b(j) \vee \neg Hom(\mathcal{X}, \ddot{\mathcal{A}}, T_a) \vee \neg Hom(\mathcal{X}, \ddot{\mathcal{A}}, T_b).
\end{gather*}
Since for any map $T : [n] \to [q]$ and for all $i < n$ there exists $a < q$ such that $T(i) = a$ and $T(j) = T(j)$ for all $j < i$, it follows that $\neg Hom(\mathcal{X}, \ddot{\mathcal{A}}, T)$. Hence, the instance $\Theta$ is unsatisfiable.

\end{proof}

Note that we have shown that the following formula is $\Sigma^{1,b}_0$.
\begin{equation}
    \begin{gathered}
        \varphi(\mathcal{X} ,\ddot{\mathcal{A}}, R, T ) = Comp_{}(R,\mathcal{X},\ddot{\mathcal{A}})\wedge  R^{[\#E_{\mathcal{X}}]}=\emptyset \longrightarrow \neg Hom(\mathcal{X},\ddot{\mathcal{A}}, T).
    \end{gathered}
\end{equation}
The upper bound for Mal'tsev CSPs then follows from Theorems \ref{Translation} and \ref{jkjkajsflageee}. 

\begin{thm}[Mal'tsev Upper Bound]
For any relational structure $\mathcal{A}$ that corresponds to an algebra with a Mal'tsev operation, there exists a polynomial-time algorithm $P$ such that, for any unsatisfiable instance $\mathcal{X}$ of CSP($\mathcal{A}$), the output $P(\mathcal{X})$ is a propositional proof in the extended Frege proof system of the propositional translation $||\neg Hom(\mathcal{X}, \mathcal{A})||$.
\end{thm}

\section{Soundness of Dalmau's algorithm in theory $V^1$}\label{alskdyf;kajdsghy;kgj}

In this section, we extend the result for Mal'tsev CSPs to GMM CSPs by formalizing Dalmau's algorithm, which solves instances of GMM CSPs in polynomial time~\cite{lmcs:2237}.

The structure of Dalmau's algorithm closely mirrors that of Mal'tsev algorithm, as described in~\cite{lmcs:2237}. Therefore, we do not repeat the formalization of most of the algorithm, as it does not introduce any novel elements relevant to the paper. Instead, we focus on the differences between the two algorithms.

\subsection{Formalization of Dalmau's algorithm in $V^1$}\label{aorgj['aigh'a]}

Fix an algebra $\A$ of size $q$ with a $(k+1)$-ary GMM operation $\varphi$, represented as a pair $([q], F)$. Then for $F$ we state
\begin{equation}\label{asdfgg}
    \begin{gathered}
        GMM_{k+1}([q], F) \iff wMap_{k+1}(\underbrace{[q],[q],\ldots,[q]}_{k+2}, F) \wedge \forall a<q\forall b <q\\
\big(F(a, b, \ldots , b) = F(b, a, \ldots , b) = \ldots = F(b, b, \ldots , a) = b   \,\,\wedge\,\,\,\,\,\\
F(b, a, \ldots , a) = F(a, b, \ldots , a) = \ldots = F(a, a, \ldots , b) = a \big) \\
\vee\\
\big( F(a,b,\ldots,b) = F(b,b\ldots ,a) = a \,\,\wedge\,\,\,\,\,\\
F(b,a,\ldots,a) = F(a,a\ldots ,b) = b\big) .
    \end{gathered}
\end{equation}
For $F$ satisfying \eqref{asdfgg} We say that $a,b$ is a minority (majority) pair in $A$ if
\begin{equation}
    \begin{gathered}
        MinPair(a,b,F)\iff F(a,b,\ldots,b) = F(b,b,\ldots,a) = a, \\
        MajPair(a,b,F) \iff F(a,b,\ldots,b) = F(b,b,\ldots,a) = b.
    \end{gathered}
\end{equation}
For any subset $R' \subseteq A^n$, we encode its signature as in Section~\ref{askfdh;qweyi} using a set $S(i, a, b)$, with the additional requirement that $a$ and $b$ form a minority pair. However, the set $R$ encoding a compact representation of $R'$ must be modified, as we now work with two distinct types of objects: elements stored in $R$ based on the signature of $R'$, and elements stored in $R$ based on projections of $R'$ onto subsets $I \subseteq \{0, \ldots, n - 1\}$ with $|I| \leq k$.

During the computation, we must be able to store and retrieve elements of the second type using unique identifiers. In most cases, these two types of objects are used in separate subroutines, with the exception of the function $nonempty$, which must have uniform access to all elements of $R$. This is why all elements - regardless of type - are stored within a single set $R$. The structure of $R$ is as follows.
\begin{equation}
    \begin{gathered}
    R(\underbrace{0}_{\text{signature type}},\underbrace{i,a,b}_{\text{signature index}},\,\,\,\,\,\,\underbrace{q,\ldots ..,q},\underbrace{r,c}_{\text{map}}),\\
    R(1,i,a,q,q,\ldots ..,q,r,c)\\
    R(\underbrace{2}_{\text{projection type |I|=2} },\,\,\,\,\,\,\,\,\,\,\,\,\underbrace{i_1,i_2,a_{i_1},a_{i_2}}_{\text{projection index}},q,\ldots .,q,r,c),\\
    \ldots \\
    R(\underbrace{k}_{\text{projection type |I|=k}},\,\,\,\,\,\,\,\,\,\,\,\,\,\underbrace{i_1,\ldots,i_k,a_{i_1},\ldots,a_{i_k}}_{\text{projection index}},r,c).
    \end{gathered}
\end{equation}
To avoid storing duplicate elements, we always require $i_1<i_2<\ldots <i_s$ for every $s\leq k$. The size of this set is polynomially bounded by 
$$b_r = \langle k, \underbrace{n,\ldots,n}_{k}, \underbrace{q,\ldots,q}_k, n,q  \rangle.
$$
The condition for $S$ to be a signature for $R$ is then reformulated as follows:
\begin{equation}
\begin{gathered}
 Sig_{CR}(S, R) \iff \forall i<n\forall a<q\forall b<q\,\,  MinPair(a,b,F) \rightarrow \\
 \big[ S(i,a,b)\longleftrightarrow
R^{[0iabq\ldots q]}(i)=a \wedge R^{[0ibaq\ldots q]}(i)=b \wedge \\
\forall j<i\,\,R^{[0iabq\ldots q]]}(j) = R^{[0ibaq\ldots q]}(j)\wedge wMap_1([n], [q], R^{[0iabq\ldots q]]})\wedge wMap_1([n], [q], R^{[0ibaq\ldots q]})\big],
\end{gathered}
\end{equation}
and for $\mathscr{R}_{\Theta}$ as follows:
\begin{equation}\label{eqj3jkr}
\begin{gathered}
   Sig_{SS}(S, \mathcal{X},\ddot{\mathcal{A}}) \iff \forall i<n\forall a<q\forall b<q\,\,  MinPair(a,b,F) \rightarrow \\\big[S(i,a,b)\longleftrightarrow
    \exists T_a,T_b\leq b_h\,\,
T_a(i)=a \wedge T_b(i)=b \wedge \forall j<i\,\, T_a(j) = T_b(j)\wedge \\
\wedge Hom(\mathcal{X},\ddot{\mathcal{A}}, T_a)\wedge Hom(\mathcal{X},\ddot{\mathcal{A}}, T_b)\big].
\end{gathered}
\end{equation}
For $R$ to be a compact representation of $\mathscr{R}_{\Theta}$ we now require:
\begin{equation}
    \begin{gathered}
        CompRep(R,\mathscr{R}_{\Theta}) \iff 
        Sig(R)=Sig(\mathscr{R}_{\Theta}) \wedge R\subseteq \mathscr{R}_{\Theta} \wedge \forall T< b_h \,\, Hom(\Theta, T)\rightarrow\\
        \forall i<n \exists g<b_r\,\, R^{[g]}(i_1) = T(i_1)\wedge\\
         \forall i_1<i_2<n\,\,\exists g<b_r\,\, R^{[g]}(i_1) = T(i_1)\wedge R^{[g]}(i_2) = T(i_2) \wedge \\
         \ldots \\
         \forall i_1<\ldots <i_k<n\,\,\exists g<b_r\,\, R^{[g]}(i_1) = T(i_1)\wedge \ldots \wedge R^{[g]}(i_k) = T(i_k).
    \end{gathered}
\end{equation}
A sequence of CSP instances $\mathcal{X}_0:=([n], E_{\mathcal{X}_0}), \ldots, \mathcal{X}_m:=([n], E_{\mathcal{X}_m}) = \mathcal{X}$ is defined exactly as in Section \ref{alskhf;shgf}. As the initial set $R^{[0]}$ we encode the set from Example \ref{++++____((JHJHJH}. Recall that $d_j = min(D_j)$.
\begin{equation}
    \begin{gathered}\label{askjfh;ashfd}
\forall i<n, \forall a,b\in D_i,\, MinPair(a,b,F)\\
R^{[0]}(0,i, a,b ,q,\ldots,q,i,a)\wedge \forall j \neq i < n \, R^{[0]}( 0,i, a,b ,q,\ldots,q,j,d_j), \\
\forall i<n\forall a\in D_i\\
R^{[0]}(1,i, a,q ,q,\ldots,q,i,a)\wedge \forall j \neq i < n \, R^{[0]}( 1,i, a, q,q,\ldots,q,j,d_j),\\
\forall i_1<i_2<n\forall a_{i_1}\in D_{i_1}\forall a_{i_2}\in D_{i_1} \\
R^{[0]}(1,i_1,i_2,a_{i_1},a_{i_2},q ,\ldots,q,i_1,a_{i_1})\wedge R^{[0]}(1,i_1,i_2,a_{i_1},a_{i_2},q,\ldots,q,i_1,a_{i_1})\wedge \\
\forall j<n,\, j\neq i_1\wedge j\neq i_2\,\, \, R^{[0]}(1,i_1,i_2,a_{i_1},a_{i_2},q ,\ldots,q,j,d_j),\\
\ldots \\
\forall i_1<\ldots <i_k<n\forall a_{i_1}\in D_{i_1},\ldots,\forall a_{i_k}\in D_{i_k} \\
R^{[0]}(k,i_1,\ldots,i_k,a_{i_1},\ldots,a_{i_k},i_1,a_{i_1})\wedge \ldots \wedge R^{[0]}(k,i_1,\ldots,i_k,a_{i_1},\ldots,a_{i_k},i_k,a_{i_k})\wedge \\
\forall j<n,\, j\neq i_1\wedge\ldots \wedge j\neq i_k\,\, \, R^{[0]}(k,i_1,\ldots,i_k,a_{i_1},\ldots,a_{i_k},j,d_j).
    \end{gathered}
\end{equation}

As previously mentioned, the structure of Dalmau's algorithm is identical to that of Mal'tsev algorithm. We do not describe in detail the execution of the functions $fixvalues_{\texttt{gmm}}$, $next_{\texttt{gmm}}$, or the group of functions collectively referred to as $nonempty_{\texttt{gmm}}$, as their roles remain exactly the same as in the Mal'tsev case.

The only subroutine that differs substantially is the computation of projections onto at most $k$ coordinates using the function $nonempty_{\texttt{gmm}}$. This subroutine appears within both $fixvalues_{\texttt{gmm}}$ and $next_{\texttt{gmm}}$. Below, it is presented as part of $fixvalues_{\texttt{gmm}}$ in lines $4-5$.
\begin{algorithm}[H]\label{asjd643}
\caption{$fixvalues_{\texttt{gmm}}(R^{[l-1]}, a_0, \ldots, a_{p-1})$}\label{aaasp}
\begin{algorithmic}[1] 
\State $U_{-1} := R^{[l-1]}$ 

\For{$j$ from $0$ to $p-1$}
   \State $U_j := \emptyset$ 
   \For{all $I = \{i_1,\ldots,i_{|I|}\}\subseteq \{0,\ldots,n-1\}$ with $|I|\leq k$ and $(b_{i_1},\ldots,b_{i_{|I|}})\in\pi_I(U_{j-1})$}
   \State $U_j :=U_j\cup nonempty_{\texttt{gmm}}^3(U_{j-1},j,i_1,\ldots,i_{|I|},\{a_j,b_{i_1},\ldots,b_{i_{|I|}}\})$
   \EndFor
        \For{all $(i, a, b) \in \mathrm{Sig}(U_{j-1})$ with $i > j$ and $a,b$ a minority pair}
        \State Let $T_a, T_b \in U_{j-1}$ witness the tuple $(i, a, b)$; if $a = b$, choose $T_a = T_b$.
\State Let $T := nonempty_{\texttt{gmm}}^4(U_{j-1}, j, i, \{(a_j, a)\})$ 
        \If{$T \neq \emptyset$} 
           \State $U_j := U_j \cup \{T, \varphi(T, T,\ldots,T,\varphi(T, T_a,\ldots,T_a, T_b)\}$ 
        \EndIf
     \EndFor
\EndFor
\\
\Return{$U_{p-1}$}
\end{algorithmic}
\end{algorithm}
According to line 4, for every $0 \leq j < p$, the function $nonempty_{\texttt{gmm}}^3$ in evoked at most
\[
\binom{n}{1}q + \binom{n}{2}q^2 + \ldots + \binom{n}{k}q^k = \sum_{i=1}^{k} \binom{n}{i} q^i = b_k
\]
times, which is polynomial in the number of variables. The set constructed during a single run of $nonempty_{\texttt{gmm}}^3$ requires at most $q^{|I|+1}$ layers, which remains polynomial for each fixed $I$. 

The set $U^{[j]}$ has the same structure as the set $R$, and in lines $4-5$ we construct the layers $U^{[j][1]}$ through $U^{[j][k]}$. Let us now examine its definition. For the purpose of formalization, the function $nonempty_{\texttt{gmm}}^3$ will be split into $k$ separate functions:
\[
nonempty_{\texttt{gmm}}^{3,|I|=1}, \quad nonempty_{\texttt{gmm}}^{3,|I|=2}, \quad \ldots, \quad nonempty_{\texttt{gmm}}^{3,|I|=k}.
\]
Then the definition is as follows. 
\begin{equation}
    \begin{gathered}
  \forall i<n\forall b<q \forall r<n\forall c<q \,\,U^{[j][1]}(i,b,q,\ldots,q,r, c) \iff \\
  \forall r'<n \exists c'<q\,\,
  U^{[j-1][1][ibq\ldots q]}(r')=c'\wedge 
  \\
nonempty_{\texttt{gmm}}^{3,|I|=1}(U^{[j-1]}, j,i,\{(a_{j},b\})(r,c)
  \\
  \forall i_1<i_2<n\forall b_{i_1}<q \forall b_{i_2}<q \forall r<n\forall c<q \,\,U^{[j][2]}(i_1,i_2,b_{i_1},b_{i_2},q,\ldots,q,r, c) \iff \\
  \forall r'<n \exists c'<q\,\,
  U^{[j-1][2][i_1i_2b_{i_1}b_{i_2}q\ldots q]}(r_1)=c_1\wedge 
  \\
nonempty_{\texttt{gmm}}^{3,|I|=2}(U^{[j-1]}, j,i_1,i_2,\{(a_{j},b_{i_1},b_{i_2}\})(r,c)\\
  \ldots \\
\forall i_1<i_2<\ldots <i_k<n, \forall b_{i_1}<q,\ldots,\forall b_{i_k}<q\forall r<n\forall c<q \,\,U^{[j][k]}(i_1,\ldots,i_k,b_{i_1},\ldots,b_{i_k},r, c) \iff \\
  \forall r'<n \exists c'<q\,\,
  U^{[j-1][k][i_1\ldots i_kb_{i_1}\ldots b_{i_k}q\ldots q]}(r_1)=c_1\wedge 
  \\
nonempty_{\texttt{gmm}}^{3,|I|=k}(U^{[j-1]}, j,i_1,\ldots,i_k,\{(a_{j},b_{i_1},\ldots,b_{i_k})\})(r,c).
    \end{gathered}
\end{equation}
The second lines in all $k$ definitions correspond to checking whether $(b_{i_1}, \ldots, b_{i_{|I|}}) \in \pi_I(U_{j-1})$ in the algorithm.

From the above discussion, it is clear that the formalization of Dalmau's algorithm can be carried out with only minor modifications to the approach used for Mal'tsev algorithm - though the details are likely to be more tedious due to the more complex definition of the representation set. Since the overall structure of the algorithm and the logical complexity of its subroutines remain essentially the same, we do not present the full formalization here, but instead state the resulting theorems.

\subsection{GMM Upper Bound}\label{alskh;alisydf;sdd}

\begin{thm}\label{jkjkajsflageee}
For any relational structure $\mathcal{A}$ that corresponds to an algebra with a GMM operation, theory $V^1$ proves the soundness of Dalmau's algorithm:
\[
V^1 \vdash \forall \mathcal{X} \, \forall \ddot{\mathcal{A}} \, \forall R  \;\; \left( \mathrm{Comp}(R, \mathcal{X}, \ddot{\mathcal{A}}) \wedge R^{[\#E_{\mathcal{X}}]} = \emptyset \longrightarrow \neg \mathrm{Hom}(\mathcal{X}, \ddot{\mathcal{A}}) \right),
\]
\end{thm}

\begin{thm}[GMM Upper Bound]
For any relational structure $\mathcal{A}$ that corresponds to an algebra with a GMM operation, there exists a polynomial-time algorithm $P$ such that, for any unsatisfiable instance $\mathcal{X}$ of CSP($\mathcal{A}$), the output $P(\mathcal{X})$ is a propositional proof in the extended Frege proof system of the propositional translation $||\neg Hom(\mathcal{X}, \mathcal{A})||$.
\end{thm}

\section{Conclusion}

    We have shown that the soundness of Mal’tsev algorithm, which solves Mal’tsev CSPs in polynomial time, can be formalized and proved in the bounded arithmetic theory $V^1$, corresponding to the extended Frege proof system. As a consequence, tautologies expressing the non-existence of a solution for unsatisfiable instances of Mal’tsev CSPs admit polynomial-size extended Frege proofs. We further extended this result to Dalmau’s algorithm, which solves CSPs defined by generalized majority-minority (GMM) polymorphisms. Thus, for both algorithms, short proofs of the translation of $\neg \mathrm{Hom}(\mathcal{X}, \mathcal{A})$ in the extended Frege system can be viewed as independent witnesses of unsatisfiability, that is, of the non-existence of a homomorphism between the corresponding relational structures.

We do not believe that the theory for Mal'tsev and GMM CSPs can be further weakened - at least not for the algorithms under consideration - since their formalization requires $\Sigma^{1,b}_1$-induction.
    
    An open question is to investigate the proof complexity of CSPs defined by algebras with few subpowers. An algebra $\mathbb{A}$ is said to have few subpowers if every subalgebra of $\mathbb{A}^n$ has a generating set of size $O(n^k)$ for some fixed $k$. The algorithm for CSPs over such algebras was significantly influenced by Dalmau's algorithm.

    Further open problems include studying the proof complexity of other tractable subclasses of CSPs, and determining whether the theory formalizing the full tractable class of CSPs, as developed in~\cite{gaysin2024}, can be weakened.

\bibliographystyle{plain}    

\bibliography{bibliography}

\begin{thebibliography}{10}

\bibitem{10.1145/3265985}
Albert Atserias and Joanna Ochremiak.
\newblock Proof complexity meets algebra.
\newblock {\em ACM Trans. Comput. Logic}, 20(1):1--46, December 2018.

\bibitem{10.1145/2677161.2677165}
Libor Barto.
\newblock Constraint satisfaction problem and universal algebra.
\newblock {\em ACM SIGLOG News}, 1(2):14–24, oct 2014.

\bibitem{barto_et_al}
Libor Barto, Andrei Krokhin, and Ross Willard.
\newblock {Polymorphisms, and How to Use Them}.
\newblock In Andrei Krokhin and Stanislav Zivny, editors, {\em The Constraint Satisfaction Problem: Complexity and Approximability}, volume~7 of {\em Dagstuhl Follow-Ups}, pages 1--44. Schloss Dagstuhl--Leibniz-Zentrum fuer Informatik, Dagstuhl, Germany, 2017.

\bibitem{Kibernetika}
V.~G. Bodnar{\v{c}}uk, L.~A. Kalu{\v{z}}nin, V.~N. Kotov, and B.~A. Romov.
\newblock Galois theory for {P}ost algebras. {I}, {II}.
\newblock {\em Kibernetika (Kiev)}, (3):1--10; ibid. 1969, no. 5, 1--9, 1969.

\bibitem{https://doi.org/10.48550/arxiv.2210.07383}
Zarathustra Brady.
\newblock Notes on csps and polymorphisms, 2022.

\bibitem{doi:10.1137/050628957}
Andrei Bulatov and V\'{\i}ctor Dalmau.
\newblock A simple algorithm for mal'tsev constraints.
\newblock {\em SIAM Journal on Computing}, 36(1):16--27, 2006.

\bibitem{articlestdfh}
Andrei Bulatov and Víctor Dalmau.
\newblock A simple algorithm for mal'tsev constraints.
\newblock {\em SIAM J. Comput.}, 36:16--27, 01 2006.

\bibitem{BULATOV200531}
Andrei~A. Bulatov.
\newblock H-coloring dichotomy revisited.
\newblock {\em Theoretical Computer Science}, 349(1):31 -- 39, 2005.
\newblock Graph Colorings.

\bibitem{8104069}
Andrei~A. {Bulatov}.
\newblock A dichotomy theorem for nonuniform csps.
\newblock In {\em 2017 IEEE 58th Annual Symposium on Foundations of Computer Science (FOCS)}, pages 319--330, 2017.

\bibitem{BurrisSankappanavar1981}
Stanley Burris and Hanamantagouda~P. Sankappanavar.
\newblock {\em A Course in Universal Algebra}.
\newblock Springer, 1981.

\bibitem{buss1}
Samuel~R. Buss.
\newblock Bounded arithmetic and propositional proof complexity.
\newblock In Helmut Schwichtenberg, editor, {\em Logic of Computation}, pages 67--121, Berlin, Heidelberg, 1997. Springer Berlin Heidelberg.

\bibitem{10.5555/1734064}
Stephen~A. Cook and Phuong Nguyen.
\newblock {\em Logical Foundations of Proof Complexity}.
\newblock Cambridge University Press, USA, 1st edition, 2010.

\bibitem{lmcs:2237}
Victor Dalmau.
\newblock Generalized majority-minority operations are tractable.
\newblock {\em Logical Methods in Computer Science}, Volume 2, Issue 4, Sep 2006.

\bibitem{1}
Tom\'{a}s Feder and Moshe~Y. Vardi.
\newblock Monotone monadic snp and constraint satisfaction.
\newblock In {\em Proceedings of the Twenty-Fifth Annual ACM Symposium on Theory of Computing}, STOC ’93, page 612–622, New York, NY, USA, 1993. Association for Computing Machinery.

\bibitem{9579073}
Azza Gaysin.
\newblock $\mathcal{H}$-colouring dichotomy in proof complexity.
\newblock {\em Journal of Logic and Computation}, 31(5):1206--1225, 2021.

\bibitem{gaysin2023proof}
Azza Gaysin.
\newblock Proof complexity of csp, submitted, arxiv: https://arxiv.org/abs/2201.00913, 2023.

\bibitem{gaysin2024}
Azza Gaysin.
\newblock Proof complexity of universal algebra in a csp dichotomy proof, submitted, arxiv: https://arxiv.org/abs/2403.06704, 2023.

\bibitem{krajicek_1995}
Jan Krajicek.
\newblock {\em Bounded Arithmetic, Propositional Logic and Complexity Theory}.
\newblock Encyclopedia of Mathematics and its Applications. Cambridge University Press, 1995.

\bibitem{krajicek2019proof}
Jan Kraj{\'\i}{\v{c}}ek.
\newblock {\em Proof Complexity}.
\newblock Encyclopedia of Mathematics and its Applications. Cambridge University Press, 2019.

\bibitem{10.1145/3402029}
Dmitriy Zhuk.
\newblock A proof of the csp dichotomy conjecture.
\newblock {\em J. ACM}, 67(5):1–78, August 2020.

\end{thebibliography}
\end{document}